\documentclass[a4paper, 11pt]{amsart}

\usepackage{a4wide}
\usepackage[english]{babel} 
\usepackage{amsmath, amssymb} 
\usepackage{mathtools}

\usepackage{tikz}
\usepackage{tikz-cd}
\usepackage{todonotes}

\usepackage[mathscr]{euscript}

\usepackage[normalem]{ulem}

\usepackage{hyperref}

\newcommand{\opcit}{\textit{op.~cit.}}

\newcommand{\Q}{\mathbb{Q}}
\newcommand{\C}{\mathbb{C}}

\newcommand{\Z}{\mathbb{Z}}
\newcommand{\F}{\mathbb{F}}

\newcommand{\A}{\mathbb{A}}

\newcommand{\red}{\mathrm{red}}
\newcommand{\Zar}{\mathrm{Zar}}
\newcommand{\cdh}{\mathrm{cdh}}
\newcommand{\mot}{\mathrm{mot}}

\newcommand{\cO}{\mathcal{O}}
\newcommand{\cA}{\mathcal{A}}
\newcommand{\cB}{\mathcal{B}}

\newcommand{\cF}{\mathcal{F}}

\newcommand{\cC}{\mathcal{C}}

\newcommand{\cE}{\mathcal{E}}
\newcommand{\cI}{\mathcal{I}}
\newcommand{\cJ}{\mathcal{J}}

\newcommand{\cL}{\mathcal{L}}

\newcommand{\fp}{\mathfrak{p}}
\newcommand{\fm}{\mathfrak{m}}

\newcommand{\cn}{\mathrm{cn}}
\newcommand{\gen}{\mathrm{gen}}

\newcommand{\cla}[1]{\prescript{\mathrm{cl}}{}{#1}}
\newcommand{\on}{\,\mathrm{on}\,}

\newcommand{\E}{\mathbb{E}}

\newcommand{\aperf}{\mathrm{aperf}}

\newcommand{\CAlg}{\mathrm{CAlg}}
\newcommand{\LSym}{\mathrm{LSym}}

\renewcommand{\E}{\mathbb{E}}
\newcommand{\Spc}{\mathrm{Spc}}

\newcommand{\lto}{\longrightarrow}

\newcommand{\mmod}{/\!\!/}

\newcommand{\id}{\mathrm{id}}
\DeclareMathOperator{\Bl}{Bl}
\DeclareMathOperator{\dBl}{dBl}
\DeclareMathOperator{\im}{im}

\DeclareMathOperator{\Pro}{Pro}

\DeclareMathOperator{\Spec}{Spec}
\DeclareMathOperator{\Sp}{Sp}

\DeclareMathOperator*{\colim}{colim}

\DeclareMathOperator{\Mod}{Mod}
\DeclareMathOperator{\Perf}{Perf}

\DeclareMathOperator{\Map}{Map}
\DeclareMathOperator{\map}{map}

\DeclareMathOperator{\Fun}{Fun}

\DeclareMathOperator{\Ind}{Ind}
\DeclareMathOperator{\fib}{fib}

\DeclareMathOperator{\TC}{TC}

\DeclareMathOperator{\THH}{THH}

\DeclareMathOperator{\HN}{HN}

\DeclareMathOperator{\QCoh}{QCoh}
\DeclareMathOperator{\Sym}{Sym}
\DeclareMathOperator{\vdim}{vdim}
\DeclareMathOperator{\Frac}{Frac}
\DeclareMathOperator{\gr}{gr}
\DeclareMathOperator{\rank}{rank}

\newcommand{\dSch}{\mathrm{dSch}}

\newcommand{\ol}[1]{\overline{#1}}

\newtheorem{thm}{Theorem}
\newtheorem*{thm*}{Theorem}
\newtheorem{cor}[thm]{Corollary}
\newtheorem*{cor*}{Corollary}
\newtheorem{lemma}[thm]{Lemma}
\newtheorem{prop}[thm]{Proposition}

\theoremstyle{definition}
\newtheorem{dfn}[thm]{Definition}

\newtheorem{ex}[thm]{Example}

\newtheorem*{dfn*}{Definition}

\theoremstyle{remark}
\newtheorem{rem}[thm]{Remark}
\newtheorem*{claim}{Claim}

\numberwithin{thm}{section}

\newtheorem*{rem*}{Remark}

\newtheorem*{rems*}{Remarks}
\newtheorem*{ex*}{Example}

\theoremstyle{plain}
\newcounter{zaehler}

\newtheorem{introthm}[zaehler]{Theorem}

\title{On pro-cdh descent on derived schemes}

\author{Shane Kelly}
\address{Graduate School of Mathematical Sciences,
University of Tokyo,
3-8-1 Komaba Meguro-ku,
Tokyo 153-8914, Japan}
\email{shanekelly@g.ecc.u-tokyo.ac.jp}

\author{Shuji Saito}
\address{Graduate School of Mathematical Sciences,
University of Tokyo,
3-8-1 Komaba Meguro-ku,
Tokyo 153-8914, Japan}
\email{sshuji.goo@gmail.com}

\author{Georg Tamme}
\address{Institut f\"ur Mathematik, Fachbereich 08, Johannes Gutenberg-Universit\"at Mainz, D-55099 Mainz, Germany}
\email{georg.tamme@uni-mainz.de}

\thanks{The first author is supported by JSPS KAKENHI Grant (19K14498) and Bilateral Joint Research Projects (120213206). The second author is supported by JSPS Grant-in-aid (B) \#20H01791 representative Shuji Saito.
The third author is partially supported by the DFG through TRR 326 (Project-ID 444845124).}

\date{\today}

\begin{document}

\begin{abstract}
Grothendieck's formal functions theorem states that the coherent cohomology of a Noetherian scheme can be recovered from that of a blowup and the infinitesimal thickenings of the center and of the exceptional divisor of the blowup. In this article, we prove an analogous descent result, called ``pro-cdh descent'', for certain cohomological invariants of arbitrary quasi-compact, quasi-separated derived schemes. Our results in particular apply to algebraic $K$-theory, topological Hochschild and cyclic homology, and the cotangent complex.

As an application, we deduce that $K_n(X) = 0$ when $n < -d$ for quasi-compact, quasi-separated derived schemes $X$ of valuative dimension $d$. This generalises Weibel's conjecture, which was originally stated for Noetherian (non-derived) $X$ of Krull dimension $d$, and proved in this form in 2018 by Kerz, Strunk, and the third author.
\end{abstract}

\maketitle

\tableofcontents

\section{Introduction}

An example of an 
\emph{abstract blowup square of schemes} is a cartesian square
\begin{equation}
	\label{abu}
\begin{tikzcd}
 E \ar[d]\ar[r, hook] & Y \ar[d, "p"] \\ 
 Z \ar[r, hook, "i"] & X 
\end{tikzcd}
\end{equation}
where $Y = \Bl_{Z}(X)$ is the blowup of a Noetherian scheme $X$ in a closed subscheme $Z$. More generally, an abstract blowup square of schemes is a cartesian square \eqref{abu}  of (potentially non-Noetherian schemes) 
where $i$ is a finitely presented closed immersion and the morphism $p$ is finitely presented, proper and an isomorphism over the open complement of $Z$ in $X$. 
Coverings $\{Z \hookrightarrow X, Y \to X\}$ for all abstract blowup squares as above together with the Nisnevich coverings generate Suslin and Voevodsky's \emph{cdh} topology \cite{SVBK}, which plays an important role in the study of motives. It is an interesting question how cohomological invariants behave under such abstract blowup squares.

Thomason \cite{ThomasonBlowup} (see also \cite[Prop.~1.5]{MR2415380}) proved that for $X$ quasi-compact, quasi-separated (qcqs, for short), $Z \hookrightarrow X$ a \emph{regular} closed immersion, and $Y$ the blowup of $X$ in $Z$, the square \eqref{abu} 
gives rise to a (homotopy) cartesian square of algebraic $K$-theory spectra and hence to a long exact sequence of algebraic $K$-groups
\[
\dots \lto K_{i}(X) \lto K_{i}(Z) \oplus K_{i}(Y) \lto K_{i}(E) \lto K_{i-1}(X) \lto \dots
\]
For general abstract blowup squares \eqref{abu}, however, this will fail.
If $X$ is Noetherian, one can repair this defect by taking infinitesimal information into account.
More precisely, if $Z(n)$ and $E(n)$ denote the $n$-th infinitesimal thickenings of $Z$ in $X$ and $E$ in $Y$, respectively, Theorem A of \cite{KST} states that the square of pro-spectra
\[
\begin{tikzcd}
 K(X) \ar[d]\ar[r] & \{K(Z(n))\}_{n} \ar[d] \\ 
 K(Y) \ar[r] & \{ K(E(n))\}_{n} 
\end{tikzcd}
\]
is  cartesian 
in a certain weak sense (see Section~\ref{sec:procdh} for a definition). In particular, it
induces a long exact sequence of pro-abelian groups. This result can be viewed as a $K$-theoretic version of Grothendieck's formal function's theorem \cite[Thm.~4.1.5]{EGAIII}. It had several important precursors \cite{MR1935844,Cortinas,MR2629600,GH,GH2,Mor,Morrow-pro-unitality}. Because of the role of abstract blowup squares in the cdh topology, such a  result is often called ``pro-cdh descent.''\footnote{We remark that it would be more appropriate to call this ``pro-cdh excision'' since it is not precisely a descent statement for some topology. However, together with Nisnevich excision, it does imply actual (\v{C}ech) descent for the pro-cdh topology of \cite{Kelly:2024aa}, see Theorem~6.1 there. For this reason, and to be in line with the existing literature we stick to the term ``pro-cdh descent.''}

This pro-cdh descent result plays a central role in the resolution \cite{KST} of Weibel's $K$-dimension conjecture,  
which states that the $K$-groups of a Noetherian scheme $X$ vanish below the negative of its Krull dimension,
 and in the development of a continuous $K$-theory of rigid spaces \cite{Morrow:2016aa,KST-non-arch,MR4637146}. 

Given the rising interest in non-Noetherian schemes, which appear naturally, for example, when working over a perfectoid base ring such as $\cO_{\C_{p}}$, it is an obvious question, whether, or in which form,  pro-cdh descent  holds in this general setting. 
The following example, a variant of which was constructed by Dahlhausen and the third author \cite{Dahlhausen} precisely for that purpose and which was independently studied by the first two authors \cite[Footnote 2]{Kelly:2024ab} with respect to the pro-cdh topology, shows that pro-cdh descent as formulated above does not hold for general abstract blowup squares of non-Noetherian schemes.
\begin{ex} \label{ex:no-pro-cdh}
Let $R$ be a valuation ring of dimension at least 2. Let $0 \subsetneq \fp \subsetneq \fm$ be prime ideals in $R$, and choose $x \in \fm\setminus \fp$ and $y \in \fp\setminus \{0\}$. Then $x^{n}$ divides $y$ for all $n\geq 1$ and for each $n$ the square 
\[
\begin{tikzcd}
 \Spec(R/(x^{n})) \ar[d, "\id"]\ar[r, hook] & \Spec(R/(y)) \ar[d] \\ 
 \Spec(R/(x^{n})) \ar[r, hook] & \Spec(R/(xy))
\end{tikzcd}
\]
is an abstract blowup square. If the induced square of $K$-theory pro-spectra would be weakly cartesian, the map $K(R/(xy)) \to K(R/(y))$ would be an equivalence. However, as $y^{2} =0$ in $R/(xy)$, $1+y$ is a unit and defines a non-trivial element in the kernel of $K_{1}(R/(xy)) \to K_{1}(R/(y))$.
\end{ex}

It turns out that the reason for this failure is that we took classical quotients when forming the infinitesimal thickenings which discards torsion information. More concretely, let $A$ be a ring and $f$ an element of $A$. Derived algebraic geometry allows us to form the derived quotients $A \mmod f^{n}$ which are certain simplicial commutative rings. Then $\pi_{0}(A \mmod f^{n})$ is the classical quotient $A/f^{n}$, whereas $\pi_{1}(A\mmod f^{n}) \cong \ker(f^{n}\colon A \to A)$ still carries the information about the $f^{n}$-torsion. We can then consider the pro-system $\{A\mmod f^{n}\}_{n}$ as a (derived) formal completion of $A$. 
For a Noetherian ring $A$, this formal completion is in fact equivalent to the usual (pro-)completion $\{A/f^{n}\}_{n}$. 
Our main theorem shows that with this modification, replacing infinitesimal neighbourhoods by ``derived infinitesimal neighbourhoods,'' pro-cdh descent does in fact hold in general. 
It is natural to formulate it in the generality of derived schemes. For a precise definition and the involved finiteness conditions, we refer to Section~\ref{sec:preliminaries}.

\begin{introthm} 
	\label{thmA}
Let $f\colon Y \to X$ be a proper, locally almost finitely presented morphism of 
quasi-compact and quasi-separated (qcqs for short) derived schemes which is an isomorphism outside a closed subset $Z \subseteq |X|$ whose open complement is quasi-compact. Denote by $X_{Z}^{\wedge}$ and $Y_{Z}^{\wedge}$ the formal completion of $X$ (respectively $Y$) along $Z$ (respectively $f^{-1}(Z)$) viewed as ind-derived schemes. Then the square of pro-spectra
\begin{equation*}
\begin{tikzcd}
 K(X) \ar[d]\ar[r] & K(X_{Z}^{\wedge}) \ar[d] \\ 
 K(Y) \ar[r] & K(Y_{Z}^{\wedge})
\end{tikzcd}
\end{equation*}
is weakly cartesian.
\end{introthm}
This result is not specific to $K$-theory. It holds more generally for every localizing invariant which is $k$-connective in the sense of \cite{LT} for some integer $k$; see Theorem~\ref{thm:pro-cdh} for the general statement. For instance, it also applies to topological Hochschild homology $\THH$, topological cyclic homology $\TC$, and rational negative cyclic homology. 

\begin{rem}
One may ask whether the same statement also holds for spectral schemes. As our proof makes essential use of the theory of derived blowups of derived schemes developed by Khan and Rydh \cite{KhanRydh}, we cannot treat this more general case with our methods.
\end{rem}

We also prove the analogous result for the cotangent complex, see Theorem~\ref{thm:procdh-cotangent}. By the arguments of \cite{ElmantoMorrow} one then obtains pro-cdh descent in the above form for motivic cohomology (Corollary~\ref{cor:pro-cdh-motivic}). We thank Matthew Morrow for suggesting this result for the cotangent complex and indicating the application to motivic cohomology.

In \cite{KST}, pro-cdh descent of $K$-theory was used to derive Weibel's conjecture on the vanishing of negative $K$-groups of a Noetherian scheme $X$ below $-\dim(X)$. It is an obvious question what we can say about negative $K$-groups if we merely assume $X$ to be qcqs; see e.g.~Morrow's Oberwolfach talk \cite{MorrowOberwolfach}. 
As the Krull dimension doesn't have good properties for general non-Noetherian schemes, one might expect that instead the valuative dimension introduced by Jaffard~\cite{Jaffard} (see \cite{EHIK} or \cite[\S 7.1]{Kelly:2024aa} for accounts) gives the correct bounds. It coincides with the Krull dimension if $X$ is Noetherian.  Indeed, we prove the following.

\begin{introthm}[Theorem~\ref{thm:Weibel-vanishing}] \label{thmB}
 Let $X$ be a qcqs spectral scheme. Then the following hold.
\begin{enumerate}
\item $K_{-i}(X) = 0$ for all $i> \vdim(X)$.
\item For all $i \geq \vdim(X)$ and any integer $r \geq 0$, the pullback map $K_{-i}(X) \to K_{-i}(\A^{r}_{X})$ is an isomorphism. In other words, $X$ is $K_{-\vdim(X)}$-regular.
\end{enumerate}
\end{introthm}
Here by a spectral scheme we mean more precisely one with connective structure sheaf, and in (2), $\A^{r}$ denotes either the flat or the smooth affine space.
\begin{rem}
We  don't know an example of a scheme $X$ with $\vdim(X) > \dim(X)$ and $K_{-\vdim(X)}(X) \not=0$.
\end{rem}

This theorem immediately implies the same vanishing for Weibel's homotopy $K$-theory $KH(X)$. In other words, we get a new proof of assertion (1) in the following theorem, which we include here for the sake of completeness. As we discuss below, this theorem is well known to the experts. 
We write $L_{\cdh}$ for the cdh sheafification functor on presheaves of spectra and $K_{\geq 0}$ for the presheaf of connective algebraic $K$-theory.

\begin{thm} \label{thm:Weibel-vanishing-KH}
Let $X$ be a qcqs scheme of finite valuative dimension. 
\begin{enumerate}
\item $KH_{-i}(X) = 0$ for all $i > \vdim(X)$.
\item The natural maps $L_{\cdh}K_{\geq 0} \to L_{\cdh}K \to KH$ are equivalences.
\end{enumerate}
\end{thm}
For schemes essentially of finite type over a  field of characteristic 0, this was first proven in \cite{Haesemeyer}. For general Noetherian $X$, (1) was first proven in \cite{KerzStrunk}  and (2) in \cite{KST}. The fact that $KH$ is a cdh sheaf is \cite{Cisinski}. In the general case, the theorem follows easily from recent results on the cdh-topology \cite{EHIK} and the $K$-theory of valuation rings \cite{KellyMorrow, KSTVorst}. In fact, the proof given under Noetherian assumptions in \cite{KellyMorrow} still works, and for the reader's convenience we reproduce this proof at the end of the paper. Alternatively, the proof of \cite{KerzStrunk} respectively \cite{KST} works mutatis mutandis, using the fact that a blowup does not increase the valuative dimension (this is not necessarily true for the Krull dimension).

\begin{rem} 
For arbitrary qcqs schemes $X$, Theorem~\ref{thm:Weibel-vanishing-KH} together with the cartesian square
\[
\begin{tikzcd}
 K \ar[d]\ar[r] & KH \ar[d] \\ 
 \mathrm{TC} \ar[r] & L_{\cdh}\mathrm{TC}, 
\end{tikzcd}
\]
$(-1)$-connectivity of topological cyclic homology on affines, and \cite{EHIK} imply the slightly weaker vanishing $K_{-i}(X)=0$  for $i>\vdim(X)+2$. Elmanto and Morrow (over fields) and  
Bouis (in general)  prove refinements of this for the motivic filtration \cite[Proof of Thm.~4.12]{ElmantoMorrow}, \cite[Prop.~5.5.4]{Bouis-thesis}.
\end{rem}

\subsection*{Structure of the arguments and the paper}

The proof of Theorem~\ref{thmA} follows the outline of \cite{KST}, 
but requires substantially more input from derived geometry.
It is reduced to the two special cases of derived blowups  and finite morphisms, respectively. This reduction is achieved by a structural result about modifications of derived schemes, Theorem~\ref{thm:structure-abu}, which we view as our main contribution here. In its proof, we need several preliminary results from derived algebraic geometry which we discuss in Section~\ref{sec:preliminaries}.

The case of Theorem~\ref{thmA} for derived blowups could essentially be proved as in \cite{KST}. We here present a stronger result with a simplified proof due to Antieau (\cite{Antieau-notes};  we reproduce his proof  in Proposition~\ref{prop:descent-derived-bu}). For the case of finite morphisms in Theorem~\ref{thmA}, compared to \cite{KST} we give a simple proof of a more general result (Proposition~\ref{prop:descent-finite-modification}). 

To prove Theorem~\ref{thmB} in Section~\ref{sec:Weibel-vanishing}, we roughly follow Kerz's approach \cite{KerzICM} combined with our pro-cdh-descent for finite morphisms. Another new input is a generalization of results of Swanson and Huneke \cite{SwansonHuneke} on reductions of ideals to the non-Noetherian setting in Subsection~\ref{subsec:reductions}.

\subsection*{Related work}

Another approach to pro-cdh descent statements has been proposed by Clausen and Scholze. It is based on condensed mathematics \cite{ScholzeCondensed} and Efimov's theory of localizing invariants of large categories \cite{Efimov:2024aa}.

In forthcoming work, the first two authors introduce a ``pro-cdh topology'' on qcqs derived schemes. It will follow from Theorem~\ref{thmA} (or \ref{thm:pro-cdh}) that every localizing invariant which is $k$-connective for some integer $k$ satisfies descent for that topology.

\subsection*{Acknowledgement}  
We thank Christian Dahlhausen for discussions about pro-cdh descent in the non-Noetherian setting which led to the variant of Example~\ref{ex:no-pro-cdh} and which was one of the starting points of this project.
We are grateful to Ben Antieau for allowing us to include his pro-excision result for derived blowups, which  is stronger than the one obtained in \cite{KST}. We thank Matthew Morrow for bringing our attention to the cotangent complex and further comments and Mauro Porta for related discussions. 
We thank David Rydh for enlightening comments on our Sections~\ref{sec:preliminaries} and \ref{sec:modifications}. In particular, he proposed to formulate the more precise version of Theorem~\ref{thm:structure-abu} including different finiteness hypotheses.
Finally, we thank the referee for a careful reading of the manuscript and useful comments.

\section{Preliminaries on derived algebraic geometry}
	\label{sec:preliminaries}
	
We freely use the language of derived algebraic geometry as developed by To{\"e}n--Vezzosi \cite{MR2394633,MR2137288}, Lurie \cite{SAG} and others. The following subsections mainly serve to fix some notation and recall some notions and facts that will be used later on. Readers familiar with derived algebraic geometry may safely skip this section and only come back when needed.

\subsection{Derived rings and schemes}
	\label{subs:derived-rings}
	
We write $\CAlg^{\Delta}$ for the $\infty$-category of simplicial commutative rings \cite[\S 25.1]{SAG} and refer to its objects as \emph{derived rings}. There is a forgetful functor $\CAlg^{\Delta} \to \CAlg^{\cn}$, where $\CAlg^{\cn}$ denotes the $\infty$-category of connective $\E_{\infty}$-algebras in spectra. This functor preserves small limits and colimits and is conservative. 

 A \emph{derived scheme} is a pair $X = (|X|, \cO_{X})$ consisting of a topological space $|X|$ and a sheaf $\cO_{X}$ of derived rings on $|X|$ such that $\cla X := (|X|, \pi_{0}\cO_{X})$ is a classical scheme and the higher homotopy sheaves $\pi_{i}\cO_{X}$ are quasi-coherent $\pi_{0}\cO_{X}$-modules. We call $\cla X$ the \emph{underlying classical scheme} or the \emph{classical truncation} of $X$.
 Derived schemes form the objects of an $\infty$-category $\dSch$, and similarly as above there is a forgetful functor from derived schemes to spectral schemes.
 There is also a functor of points approach to derived schemes. In other words, there is a fully faithful functor
 \[
 \dSch \to \Fun(\CAlg^{\Delta}, \Spc), 
 \]
 which coincides with the Yoneda embedding on affine derived schemes \cite[\S 1.6]{SAG}.

A map $f\colon X \to Y$ of derived schemes is called \emph{proper, a closed immersion, affine, or finite}, respectively, if the underlying map of classical schemes $\cla f$ has the corresponding property. If $U \subseteq |X|$ is an open subset, then $(U, \cO_{X}|_{U})$ is itself a derived scheme. Such a derived scheme is called an \emph{open subscheme of X} which we simply denote by $U$.

\subsection{Quasi-coherent modules}
	\label{subs:qcoh}

If $A$ is a derived ring, we write $\Mod(A)$ for the symmetric monoidal $\infty$-category of $A$-modules (in spectra). For a derived scheme $X$, we denote by $\QCoh(X)$ the category of quasi-coherent sheaves on $X$. These are stable $\infty$-categories with canonical t-structures, and we denote their connective part by $\Mod(A)^{\cn}$ and $\QCoh(X)^{\cn}$, respectively. Moreover, they only depend on the underlying spectrum or spectral scheme, respectively. If $X=\Spec(A)$ is affine, we have an equivalence $\QCoh(X) = \Mod(A)$. 

If $Z \subseteq |X|$ is a closed subset, we denote by $\QCoh(X \on Z)$ the full subcategory of $\QCoh(X)$ spanned by those quasi-coherent sheaves which are supported on $Z$.
If $X$ is quasi-compact and quasi-separated (qcqs, for short) and the open complement of $Z$ is quasi-compact, then $\QCoh(X \on Z)$ is compactly generated and its compact objects coincide with the perfect ones, \cite[Prop.~9.6.1.1]{SAG}, \cite[Prop.~A.9]{CMNN}. Here a quasi-coherent sheaf on $X$ is called \emph{perfect} if and only its restriction to each affine open subscheme $U=\Spec(A) \subseteq X$ belongs to the smallest thick stable subcategory of $\QCoh(U) = \Mod(A)$ containing $A$. We write $\Perf(X)$ and $\Perf(X  \on Z)$ for the corresponding subcategories.
The arguments used to prove this also apply to the case of connective sheaves:

\begin{lemma}\label{lem;QCohZXcn}
Let $X$ be a qcqs derived scheme, and let $Z \subseteq |X|$ be a closed subset with quasi-compact open complement. 
Then $\QCoh(X \on Z)^{\cn}$ is compactly generated and the inclusion $\QCoh(X \on Z)^{\cn} \hookrightarrow \QCoh(X)^{\cn}$ preserves compact objects. An object of $\QCoh(X \on Z)^{\cn}$ is compact if and only if it is perfect (as an object of $\QCoh(X)$).
\end{lemma}
\begin{proof}
The same argument as in the proof of \cite[Prop.~A.9]{CMNN} proves the first two claims:
In case $X$ is affine, these  follow from \cite[Prop.~7.1.1.12(e)]{SAG} and the fact that 
$\QCoh(X)^{\cn}$ is compactly generated by \cite[Prop.~9.6.1.2]{SAG}.
The reduction of the global case to the local case is done by 
\cite[Ex.~10.3.0.2 (4), Prop.~10.3.0.3, Th.~10.3.2.1 (b)]{SAG}. It then follows that an object of $\QCoh(X \on Z)^{\cn}$ is compact if and only if its image in $\QCoh(X)^{\cn}$ is compact. The compact objects in the latter category coincide with the perfect, connective $\cO_{X}$-modules by \cite[Prop.~9.6.1.2.]{SAG} again.
\end{proof}

\subsection{Free algebras}
	\label{subs:free-algebras}

If $A$ is a derived ring, we denote the category of derived $A$-algebras by $\CAlg^{\Delta}_{A}$. Its objects are derived rings $B$ together with a map of derived rings $A \to B$.  There is an obvious forgetful functor $\CAlg^{\Delta}_{A} \to \Mod(A)^{\cn}$ which admits a left adjoint, the free algebra functor, $\LSym^{*}_{A}\colon \Mod(A)^{\cn} \to \CAlg_{A}^{\Delta}$. If $M$ is a connective $A$-module, the underlying $A$-module of $\LSym^{*}_{A}(M)$ is the direct sum $\bigoplus_{n\geq 0} \LSym^{n}_{A}(M)$, where $\LSym^{n}_{A}(M)$ is the $n$-th derived symmetric power of $M$ as studied for instance by Quillen \cite[Constr.~25.2.2.6]{SAG}, whence the notation. 

These constructions globalise: 
For a derived scheme $X$, write $\CAlg^{\Delta}_{\cO_{X}}$ for the $\infty$-category of sheaves of derived rings on $|X|$ equipped with a map from $\cO_{X}$ such that the underlying sheaf of $\cO_{X}$-modules is quasi-coherent. There is an adjunction
\[
\LSym^{*}_{\cO_{X}} \colon \QCoh(X)^{\cn}  \rightleftarrows \CAlg^{\Delta}_{\cO_{X}} : \mathrm{forget},
\]
the forgetful functor is conservative and preserves sifted colimits. Consequently, $\LSym^{*}_{\cO_{X}}$ sends compact objects to compact objects.

For $\cA \in \CAlg^{\Delta}_{\cO_{X}}$ one can form the relative spectrum $\Spec(\cA)$ which comes with an affine morphism $\Spec(\cA) \to X$.

\subsection{Finiteness conditions}
	\label{subs:finiteness-conditions}

A map of derived rings $A \to B$ is called \emph{locally of finite presentation} if $B$ is a compact object of $\CAlg^{\Delta}_{A}$. It is called \emph{almost of finite presentation} if $B$ is an almost compact object $\CAlg^{\Delta}_{A}$, i.e.~each truncation $\tau_{\leq n}B$ is a compact object of $\tau_{\leq n}\CAlg^{\Delta}_{A}$; see \cite[\S 3.1]{LurieThesis} and \cite[\S 4.1]{SAG} for the analog notions for $\E_{\infty}$-algebras.
It turns out that $A \to B$ is almost of finite presentation if and only if the underlying map of $\E_{\infty}$-algebras is almost of finite presentation; the analog for being locally of finite presentation is wrong. 

For completeness, we also mention that for an integer $n\geq 0$, there is a notion of \emph{finite generation to order n}, see \cite[Def.~4.1.1.1]{SAG}, and a morphism is almost finitely presented if and only if it is of finite generation to order $n$ for all $n$. 

These finiteness conditions are stable under base change: If $A \to B$ is locally or almost of finite presentation or of finite generation to order $n$ and $A \to A'$ is an arbitrary map, then also  $A' \to B\otimes_{A} A'$ is locally or almost of finite presentation or of finite generation to order $n$, respectively \cite[Rem.~7.2.4.28]{HA}, \cite[Prop.~4.1.3.2]{SAG}. They are also stable under composition \cite[Rem.~7.2.4.29, Cor.~7.4.3.19]{HA}, \cite[Prop.~4.1.3.1]{SAG}.

If $A$ is Noetherian, i.e.~$\pi_{0}(A)$ is Noetherian in the classical sense and all higher homotopy groups are finitely generated $\pi_{0}(A)$-modules, then a derived $A$-algebra $B$ is almost of finite presentation if and only if $B$ is Noetherian and $\pi_{0}(B)$ is a classically finitely generated $\pi_{0}(A)$-algebra \cite[Prop.~3.1.5]{LurieThesis} or \cite[Prop.~7.2.4.31]{HA}.

A map $f\colon Y \to X$ of derived schemes is \emph{called locally of finite presentation} or \emph{locally almost of finite presentation} (\emph{lafp}, for short) if for all affine open subschemes $U=\Spec(A) \subseteq X$ and $V=\Spec(B) \subseteq Y$ with $f(V) \subseteq U$ the induced morphism $A \to B$ is locally of finite presentation or almost of finite presentation, respectively.
As in the affine case, these notions are stable under base change and composition, and there is a characterization in the Noetherian case.

For example, if $\cF$ is a perfect, connective $\cO_{X}$-module, i.e.~a compact object of $\QCoh(X)^{\cn}$, then  $\Spec(\LSym^{*}_{\cO_{X}}(\cF))$ is locally of finite presentation over $X$. 
Using this observation, we prove the following lemma.

\begin{lemma}\label{lem:closed-subscheme}
Let $X$ be a qcqs derived scheme, and let $U \subseteq X$ be a qc open subset with complement $Z$. Let $i\colon Y \to X$ be a closed immersion which is an isomorphism over $U$ and such that the underlying map of classical schemes $\cla Y \to \cla X$ is finitely presented. Then there exists a factorization $Y \to Y' \to X$ of $i$ such that 
\begin{enumerate}
\item the morphism $Y \to Y'$ is an isomorphism on underlying classical schemes,
\item the morphism $Y' \to X$ is a closed immersion locally of finite presentation and an isomorphism over $U$.
\end{enumerate}
\end{lemma}

\begin{proof}
Let $\cI = \fib (\cO_{X} \to i_{*}\cO_{Y})$. As $i$ is a closed immersion and an isomorphism over $U$, we have $\cI \in \QCoh(X \on Z)^{\cn}$. As the morphism of classical schemes $\cla Y \to \cla X$ is classically of finite presentation, it follows that the image $\cJ$ of $\pi_{0}(\cI)$ in $\pi_{0}(\cO_{X})$ is of finite type. By Lemma~\ref{lem;QCohZXcn}, we can write $\cI$ as a filtered colimit $\cI = \colim_{\alpha} \cI_{\alpha}$ where each $I_{\alpha} \in \Perf(X \on Z)^{\cn}$. 
As $\pi_{0}(-)$ commutes with filtered colimits, we have $\colim_{\alpha} \pi_{0}(\cI_{\alpha}) = \pi_{0}(\cI)$. As $\pi_{0}(\cI) \to \cJ$ is surjective and $\cJ$ is of finite type, there exists an index $\alpha$ such that the induced map $\pi_{0}(\cI_{\alpha}) \to \cJ$ is surjective.
By construction, we have the following commutative diagram in $\QCoh(X)^{\cn}$.
\[
\begin{tikzcd}
 \cI_{\alpha} \ar[d]\ar[r] & 0 \ar[d] \\ 
 \cO_{X} \ar[r] & i_{*}\cO_{Y} 
\end{tikzcd}
\]
By adjunction, this induces a commutative diagram
\[
\begin{tikzcd}
 \LSym^{*}_{\cO_{X}}(\cI_{\alpha}) \ar[d]\ar[r] & \cO_{X} \ar[d] \\ 
 \cO_{X} \ar[r] & i_{*}\cO_{Y} 
\end{tikzcd}
\]
in $\CAlg^{\Delta}_{\cO_{X}}$. We define $\cA \in \CAlg^{\Delta}_{\cO_{X}}$ to be the tensor product
$\cO_{X} \otimes_{\LSym^{*}_{\cO_{X}}(\cI_{\alpha})} \cO_{X}$. The above diagram classifies a morphism $\cA \to i_{*}\cO_{Y}$ in $\CAlg^{\Delta}_{\cO_{X}}$. 
We claim that it induces an isomorphism on $\pi_{0}$. Indeed, we compute\footnote{here the $\otimes^{\heartsuit}$ indicates the underived tensor product and $\Sym$ denotes the classical symmetric algebra}
\begin{align*}
\pi_{0}(\cA) &\cong \pi_{0}(\cO_{X}) \otimes^{\heartsuit}_{\pi_{0}(\LSym^{*}_{\cO_{X}}(\cI_{\alpha}))} \pi_{0}(\cO_{X}) \\
&\cong \pi_{0}(\cO_{X}) \otimes^{\heartsuit}_{\Sym_{\pi_{0}(\cO_{X})}(\pi_{0}(\cI_{\alpha}))} \pi_{0}(\cO_{X}) \\
&\cong \pi_{0}(\cO_{X}) / \im(\pi_{0}(\cI_{\alpha}) \to \pi_{0}(\cO_{X})) \\
&= \pi_{0}(\cO_{X}) /\cJ \\
&\cong \pi_{0}(i_{*}\cO_{Y}).
\end{align*}
We set $Y' = \Spec(\cA)$. By construction, we get the factorization $Y \to Y' \to X$ of $i$. The above computation shows that $Y \to Y'$ is an isomorphism on underlying classical schemes. Moreover, as $I_{\alpha}$ is supported on $Z$, $Y' \to X$ is an isomorphism over $U$. Finally, as $I_{\alpha}$ is perfect, it follows that $Y' \to X$ is locally of finite presentation, as desired.
\end{proof}

\subsection{Formal completion}
	\label{subs:formal-completion}

Let $X$ be a qcqs derived scheme and let $Z \subseteq |X|$ be  a closed subset whose open complement $|X| \setminus Z$ is  quasi-compact. The \emph{formal completion} $X_{Z}^{\wedge}$ is an ind-object of derived schemes with a map $i\colon X_{Z}^{\wedge} \to X$. It is determined by the following universal property: For any derived ring $R$,  composition with $i$ induces an equivalence of $\Map_{\Ind(\dSch)}(\Spec(R), X_{Z}^{\wedge})$ with the union of components of $\Map_{\dSch}(\Spec(R), X)$ consisting of those morphisms $\Spec(R) \to X$ that  set-theoretically factor through $Z$; see \cite[Prop.~6.5.5]{MR3220628} for the existence of $X_{Z}^{\wedge}$ as an ind-derived scheme. 
More concretely, we can write $X_{Z}^{\wedge}$ as the ind-system 
\begin{equation}
	\label{eq:completion}
X_{Z}^{\wedge} = \{ Z' \}_{Z' \hookrightarrow X}
\end{equation}
of (a small cofinal subsystem of) all closed immersions of derived schemes $Z' \hookrightarrow X$ with $|Z'| = Z$.

If $X=\Spec(A)$ is affine, there exist finitely many elements $f_{1}, \dots, f_{r} \in \pi_{0}(A)$ whose zero locus is $Z$. In this case, $X_{Z}^{\wedge}$ can also be represented by
\begin{equation}
	\label{eq:completion-affine}
X_{Z}^{\wedge} = \{ \Spec(A\mmod f_{1}^{\alpha}, \dots, f_{r}^{\alpha} ) \}_{\alpha\geq 1} 
\end{equation}
where the symbol $\mmod$ indicates the derived quotient, i.e.~the (derived) tensor product $A \otimes_{\Z[t_{1}, \dots, t_{r}]} \Z$ where the maps send $t_{i}$ to $f_{i}^{\alpha}$ and $0$, respectively; see \cite[Prop.~6.1.1]{LurieThesis} or \cite[Lemma~8.1.2.2]{SAG}.

It follows immediately from the universal property that the formation of the derived completion commutes with base change: If $f\colon Y \to X$ is a quasi-compact map, then $Y^{\wedge}_{f^{-1}(Z)} \simeq X_{Z}^{\wedge} \times_{X} Y$. We therefore also write $Y_{Z}^{\wedge}$ instead of $Y_{f^{-1}(Z)}^{\wedge}$.

If $X$ is a Noetherian classical scheme, then the formal completion is itself classical, equal to the classical formal completion.  This follows for example from \cite[Lemma~17.3.5.7]{SAG}.
 
\subsection{Ample line bundles}
	\label{subs:ample}
	
Ample line bundles on Noetherian derived schemes have been studied by Annala \cite{AnnalaBase}. We need some of the results in the more general setting of qcqs derived schemes. These are certainly well-known, the proofs are essentially the same as for classical schemes.

Let $X$ be a qcqs derived scheme. A line bundle $\cL$ on $X$ is called \emph{ample} if for any point $x \in X$ there exists an $n \geq 1$ and a global section $s \in \pi_{0}\Gamma(X, \cL^{\otimes n})$ such that the non-vanishing locus $X_{s}$ of $s$ is affine and contains $x$. If $f \colon X \to Y$ is a morphism of derived schemes, then $\cL$ is called \emph{$f$-ample} if for every affine open subscheme $U \subseteq Y$ the restriction of $\cL$ to $f^{-1}(U)$ is ample. 

Let now $\cL$ be any line bundle on $X$ and  $s \in \pi_{0}\Gamma(X,\cL)$ a global section. We view the latter as a map $s\colon \cO_{X} \to \cL$. If $\cF$ is any quasi-coherent sheaf on $X$, we get a diagram
\begin{equation}
	\label{eq:line-bundle-section}
	\cF \xrightarrow{\otimes s} \cF \otimes \cL \xrightarrow{\otimes s} \cF \otimes \cL^{\otimes 2} \xrightarrow{\phantom{\otimes s}} \dots
\end{equation}
In the case of classical schemes, the following lemma is standard. For Noetherian derived schemes, it is \cite[Lemma~2.6]{AnnalaBase}.

\begin{lemma}\label{lemma:section-standard-open}
In the above situation, there is a canonical equivalence
\[
\colim_{n} \Gamma(X, \cF \otimes \cL^{\otimes n}) \simeq \Gamma(X_{s}, \cF).
\]
\end{lemma}

\begin{proof}
If we take sections over $X_{s}$ in \eqref{eq:line-bundle-section}, then all maps become equivalences.
Thus the restriction maps induce a map 
\[
\colim_{n} \Gamma(X,  \cF \otimes \cL^{\otimes n}) \to \colim_{n} \Gamma(X_{s}, \cF \otimes \cL^{\otimes n}) \simeq \Gamma(X_{s}, \cF).
\]
We claim that this map is an equivalence. As $X$ is qcqs, a standard induction reduces us to the case where $X=\Spec(A)$ is affine and $\cL$ is trivial. So we may assume $X = \Spec(A)$, $\Gamma(X, \cL) = A$, and $\cF$ corresponds to the $A$-module $M=\Gamma(X, \cF)$. In this case, the colimit in question identifies with $M[s^{-1}]$ which also identifies with $\Gamma(X_{s}, \cF)$ as $\cF$ is quasi-coherent.
\end{proof}

Exactly as in \cite[Lemma~2.11]{AnnalaBase}, this can be used to prove the following lemma.

\begin{lemma}\label{lemma:rel-aample-local-on-base}
Let $f\colon X \to Y$ be a morphism of qcqs derived schemes, $\cL$ a line bundle on $X$, and $(U_{i})_{i\in I}$ an open covering of $Y$.  Write $f_{i}$ for the restricted morphism $f^{-1}(U_{i}) \to U_{i}$. Then the following are equivalent:
\begin{enumerate}
\item $\cL$ is $f$-ample;
\item for every $i\in I$, the restriction $\cL|_{f^{-1}(U_{i})}$ of $\cL$ is $f_{i}$-ample. \hfill $\square$
\end{enumerate}
\end{lemma}

The existence of an ample line bundle implies the resolution property in the following form:

\begin{lemma}
	\label{lem:resolution-property}
Let $X$ be a qcqs derived scheme which carries an ample line bundle, and let  $\cF$ be a connective, quasi-coherent $\cO_{X}$-module such that $\pi_{0}\cF$ is of finite type as a $\pi_{0}\cO_{X}$-module. Then there exists a vector bundle, i.e.~a locally free $\cO_{X}$-module of finite rank, $\cE$ on $X$ together with a map $\cE \twoheadrightarrow \cF$ which is surjective on $\pi_{0}$.
\end{lemma}

\begin{proof}
Choose an ample line bundle $\cL$. Replacing $\cL$ by an appropriate tensor power, we may assume that there exist finitely many global sections $s_{i} \in \pi_{0}\Gamma(X, \cL)$, whose non-vanishing loci $X_{s_{i}}$ form an affine covering of $X$. As each $X_{s_{i}}$ is affine and $\cF$ is quasi-coherent, we have isomorphisms $\pi_{0}\Gamma(X_{s_{i}}, \cF) = \pi_{0}\Gamma(X_{s_{i}}, \pi_{0}\cF)$ and similarly for $\cO_{X}$ in place of $\cF$. The assumption that  $\pi_{0}\cF$ is a $\pi_{0}\cO_{X}$-module of finite type then implies that each $\pi_{0}\Gamma(X_{s_{i}}, \cF)$ is a finitely generated $\pi_{0}\Gamma(X_{s_{i}}, \cO_{X})$-module. Choose finitely many generators $m_{ij} \in \pi_{0}\Gamma(X_{s_{i}}, \cF)$. By Lemma~\ref{lemma:section-standard-open} there exists an integer $N$ such that all the $m_{ij}$ extend to global sections of $\cF \otimes \cL^{\otimes N}$. These give rise to a map $\cE := \bigoplus_{ij} \cL^{\otimes (-N)} \to \cF$ which is surjective on $\pi_{0}$ by construction.
\end{proof}

\subsection{Quasi-smooth closed immersions and derived blowups}
	\label{subs:derived-blowups}

Derived blowups were first introduced in \cite{KST} for affine schemes in order to prove pro-cdh descent of algebraic $K$-theory on Noetherian schemes. They were then systematically studied and developed much further by Khan, Rydh, and Hekking \cite{KhanRydh,Hekking:2021aa}.

Let $X$ be a derived scheme, and let $Z \hookrightarrow X$ be a \emph{quasi-smooth closed immersion} \cite[2.3.6]{KhanRydh}, i.e., Zariski locally on $X$, $Z \hookrightarrow X$ is the derived pullback of the map $\{0\} \hookrightarrow \A^{r}_{\Z}$ for some $r$ and some morphism $X \to \A^{r}_{\Z}$. Equivalently, Zariski locally on $X$, $Z \hookrightarrow X$ is given by $\Spec(A\mmod f_{1}, \dots, f_{r}) \hookrightarrow \Spec(A)$ for suitable elements $f_{i} \in \pi_{0}(A)$. The number $r$ is called the \emph{virtual codimension} of the closed immersion.
Then one can form the \emph{derived blowup} $p\colon \dBl_{Z}(X) \to X$ \emph{of} $X$ \emph{in} $Z$ (or \emph{with center} $Z$) which is characterized by a universal property: it classifies virtual Cartier divisors on derived $X$-schemes; see \cite[\S 4.1]{KhanRydh} for details. The construction of the derived blowup commutes with arbitrary base change. Locally, if $Z$ is the derived pullback of $\{0\} \hookrightarrow \A^{r}_{\Z}$ along a map $X \to \A^{r}_{\Z}$, then $\dBl_{Z}(X)$ is the derived pullback of the classical blowup $\Bl_{\{0\}}(\A^{r}_{\Z}) \to \A^{r}_{\Z}$. By \cite[Thm.~4.1.5(v)]{KhanRydh}, the morphism $p\colon \dBl_{Z}(X) \to X$ is quasi-smooth, i.e.~Zariski locally factors as a quasi-smooth closed immersion followed by a smooth morphism, and in particular is locally of finite presentation.
Clearly, $p$ is an isomorphism outside $Z$ and proper. 

It follows from the description of the underlying classical scheme of $\dBl_{Z}(X)$ in \cite[Thm.~4.1.5(vii)]{KhanRydh} that there is always a closed immersion $\Bl_{\cla Z}(\cla X) \hookrightarrow \cla\dBl_{Z}(X)$ and this is an isomorphism over $\cla U$ where $U$ is the open complement of $Z$.

The derived blowup $\dBl_{Z}(X)$ carries a canonical line bundle $\cO(1)$  (the ideal sheaf defining the universal virtual Cartier divisor) and this line bundle is $p$-ample.
 Indeed, by Lemma~\ref{lemma:rel-aample-local-on-base} we may work affine-locally on $X$ and hence assume that $p\colon \dBl_{Z}(X) \to X$ is the pullback of the classical blowup $\Bl_{\{0\}}(\A^{n}_{\Z}) \to \A^{n}_{\Z}$, and the line bundle $\cO(1)$ is the pullback of the classical canonical line bundle on $\Bl_{\{0\}}(\A^{n})$ which is ample. As being relatively ample is stable under base change (the proof in \cite[Prop.~2.12]{AnnalaBase} works in general), it follows that $\cO(1)$ is $p$-ample. 

In presence of an ample line bundle, every closed subset with quasi-compact complement is the support of a quasi-smooth closed subscheme. More precisely, we have the following lemma.

\begin{lemma}\label{lem:existence-quasi-smooth-closed-subscheme-structure}
Let $X$ be a qcqs derived scheme which carries an ample line bundle. Let $Z_{0} \hookrightarrow \cla X$ be a classically finitely presented closed subscheme. Then there exists a quasi-smooth closed subscheme $Z \hookrightarrow X$ whose classical truncation is $Z_{0}$.
\end{lemma}

\begin{proof}
This is very similar to \cite[Constr.~A.2.2]{BachmannCatMilnor}. Let $J = \fib(\cO_{X} \to \cO_{Z_{0}})$. As $\cO_{Z_{0}}$ is a discrete, quasi-coherent sheaf and $\cO_{X} \to \cO_{Z_{0}}$ is surjective on $\pi_{0}$, the sheaf $J$ is connective, quasi-coherent, and $\pi_{0}J$ is the ideal sheaf defining $Z_{0}$ in $\cla X$. As $Z_{0} \hookrightarrow \cla X$ is classically of finite presentation, $\pi_{0}J$ is a $\pi_{0}\cO_{X}$-module of finite type. Hence Lemma~\ref{lem:resolution-property} implies the existence of a vector bundle $\cE$ on $X$ and a map $\cE \to \cJ$ which is surjective on $\pi_{0}$. Let $V(\cE)$ be the geometric vector bundle $V(\cE) := \Spec (\LSym^{*}_{\cO_{X}}(\cE))$ over $X$. 
The composition $\cE \to \cJ \to \cO_{X}$ defines a section $s\colon X \to V(\cE)$, and we form the fibre product
\[
\begin{tikzcd}
 Z \ar[d, hook, "i"]\ar[r] & X \ar[d, hook, "0"] \\ 
 X \ar[r, "s"] & V(\cE). 
\end{tikzcd}
\]
Equivalently, we have
\[
Z = \Spec ( \cO_{X} \otimes_{\LSym^{*}_{\cO_{X}}(\cE)} \cO_{X} ).
\]
By construction, $i \colon Z \hookrightarrow X$ is quasi-smooth. The same computation as in the proof of Lemma~\ref{lem:closed-subscheme} shows that $\cla Z = Z_{0}$, as desired.
\end{proof}

\subsection{Pushouts of derived schemes}
	\label{subs:pushouts}
	
We also need certain pushouts of derived schemes. These have been studied in \cite{GRII}:

\begin{lemma}
\label{lem:nilOpenPushout}
Let $i\colon Y_1 \to Y_1'$ be a closed immersion of derived schemes which  is an isomorphism on underlying topological spaces, and let $f\colon Y_1 \to Y_2$ be an affine map in $\dSch$.
Then the following hold.
\begin{enumerate}
 \item \label{lem:nilOpenPushout:1} The pushout 
 square
 \[ 
 \begin{tikzcd}
 Y_1 \ar[r, "i"] \ar[d, "f"] & Y_1' \ar[d, "f'"] \\
 Y_2 \ar[r, "i'"] & Y_2 \sqcup_{Y_1} Y_1'
 \end{tikzcd}
 \]
 exists in $\dSch$. Write $Y_{2}' := Y_2 \sqcup_{Y_1} Y_1'$.
The map $Y_2 \to Y_2'$ is a closed immersion and  an isomorphism on underlying topological spaces. In particular, if $Y_2$ is qcqs, then so is $Y_2'$.

\item \label{lem:nilOpenPushout:2} 
For an affine open subscheme $U_2 \subseteq Y_2$ with $U_1:=f^{-1}(U_2) \subseteq Y_1$, and the corresponding open subschemes 
$U_i' \subseteq   Y_i'$ ($i=1,2$), the map
\[ 
U_2\sqcup_{U_1} U_1' \to U_2'
\]
is an isomorphism.

\item \label{lem:nilOpenPushout:3}  If $f$ is an open immersion, then so is  
$f'$.

\item \label{lem:nilOpenPushout:4} 
Assume  
$i$ exhibits  $Y_1$ as the underlying classical scheme of $Y'_{1}$. Then $i'\colon Y_{2} \to Y_{2}'$ is an isomorphism on underlying classical schemes.
\end{enumerate}
\end{lemma}
\begin{proof}
Except for \eqref{lem:nilOpenPushout:4}, this is \cite[Ch.~1, Cor.~1.3.5]{GRII}.
By \eqref{lem:nilOpenPushout:2}, we may assume $Y_i=\Spec(A_i)$ ($i=1,2$) and $Y_1'=\Spec(A_1')$ are affine and $A_1=\pi_0(A_1')$.
By the construction of the pushout, $Y_2'=\Spec(A_2')$ where $A_2'=A_2\times_{A_1} A_1'$ is the pullback in derived rings.
We thus have an exact sequence of homotopy groups
\[ 
\pi_1(A_1) \to \pi_0(A_2')\to \pi_0(A_2) \oplus \pi_0(A_1')\to \pi_0(A_1),
\]
which implies \eqref{lem:nilOpenPushout:4}  as  $\pi_1(A_1)=0$ and $ \pi_0(A_1') \cong \pi_0(A_1)$.
\end{proof}

\section{Modifications of derived schemes}
	\label{sec:modifications}
	
In the following, $X$ always denotes a qcqs derived scheme.

\begin{dfn}
Let $U \subseteq X$ be a quasi-compact open subscheme. A \emph{$U$-modification} of $X$ is a proper morphism $f \colon Y \to X$ which is an isomorphism over $U$. A \emph{closed $U$-modification} is a $U$-modification which is a closed immersion. 
\end{dfn}

Note that we do not assume that a $U$-modification induces a bijection of the set of generic points.
For example, if $Z \hookrightarrow X$ is a quasi-smooth closed immersion with $|Z| \cap |U| = \emptyset$, then the derived blowup $\Bl_{Z}(X) \to X$ is a $U$-modification which is moreover lafp (see \ref{subs:derived-blowups}). In fact, derived blowups and lafp closed $U$-modifications generate all lafp $U$-modifications in the following sense:

\begin{thm}
	\label{thm:structure-abu}
Assume that $X$ carries an ample line bundle. Let $U \subseteq X$ be a quasi-compact open subscheme. 
Let $f \colon Y \to X$ be a $U$-modification of $X$, whose underlying map of classical schemes $\cla Y \to \cla X$ is classically finitely presented.
Then there exists a commutative diagram
\[
\begin{tikzcd}
 Y' \ar[d, hook, "h"']\ar[r, "g"] & Y \ar[d, "f"] \\ 
 \widetilde X \ar[r, "p"] & X 
\end{tikzcd}
\]
where $g$ is  $U$-modification of $Y$, $h$ is a closed $U$-modification,  $p$ is a derived blowup with center set-theoretically contained in  $X \setminus U$, and the induced map $(h,g)\colon Y' \to \widetilde X \times_{X} Y$ is locally of finite presentation. In particular, $g$ is locally of finite presentation, and, if $f$ is locally of finite presentation (or lafp, or locally of finite generation to order $n$, respectively), then $h$ is locally of finite presentation (or lafp, or locally of finite generation to order $n$).

Moreover, $Y'$ carries a $(p\circ h)$-ample line bundle.
\end{thm}
 
\begin{proof}
By assumption, the morphism of classical schemes $\cla f\colon\cla Y \to \cla X$ is proper and classically of finite presentation.
It is also an isomorphism over $\cla U$. 
By \cite[Cor.~5.7.12]{Raynaud-Gruson} (or \cite[Tag 081T]{stacks}), there exists a $\cla U$-admissible blowup\footnote{i.e., a blowup in a finitely presented closed subscheme of $\cla X$ which is set-theoretically contained in the complement of $\cla U \subseteq \cla X$} $Y_{0} \to \cla X$ such that the morphism $Y_{0} \to \cla X$ factors through $\cla Y \to \cla X$.
For the constructions to come, we need the open immersion $\cla U \to Y_{0}$ to be affine. As this is not necessarily the case, we make a further blowup: As $\cla U$ is quasi-compact, there exists a finitely presented closed subscheme $T_{0} \hookrightarrow Y_{0}$ whose underlying topological space is $Y_{0} \setminus \cla U$: This follows for instance from \cite[Cor.~6.9.15]{EGAInew} which allows us to write the ideal sheaf of some closed subscheme of $Y_{0}$ with support $Y_{0}\setminus \cla U$ as a filtered colimit of quasi-coherent sub-ideal sheaves of finite type. By quasi-compactness of $\cla U$, the closed subscheme defined by one of them will have support $Y_{0} \setminus \cla U$.
Let $Y_{1} \to Y_{0}$ be the blowup of $Y_{0}$ in $T_{0}$.
Then the canonical open immersion $\cla U \to Y_{1}$ is affine as its complement is a Cartier divisor. As the composite of two $\cla U$-admissible blowups is a $\cla U$-admissible blowup \cite[Tag 080L]{stacks}, the composition  $Y_{1} \to Y_{0} \to \cla X$ is a $\cla U$-admissible blowup, say $Y_{1} = \Bl_{S_{0}}(\cla X)$ for some classically finitely presented closed subscheme $S_{0} \hookrightarrow \cla X$.

As $X$ carries an ample line bundle, Lemma~\ref{lem:existence-quasi-smooth-closed-subscheme-structure} implies the existence of a quasi-smooth closed immersion of derived schemes $S \hookrightarrow X$ whose classical truncation is $S_{0} \hookrightarrow \cla X$.
Let $p\colon \widetilde X \to X$ be the derived blowup of $X$ in $S$. 
Then there is a canonical closed immersion $Y_{1}\hookrightarrow  \cla {\widetilde X}$, which is an isomorphism over $\cla U$. Note that $Y_{1}$ need not be classically of finite presentation over $X$. 
However, using \cite[Cor.~6.9.15]{EGAInew} again we see that the closed immersion $Y_{1} \hookrightarrow \cla {\widetilde X}$ can be written as a cofiltered limit of classically finitely presented closed immersions $Y_{\alpha} \to \cla {\widetilde X}$ all of which are isomorphisms over $\cla U$. 
As $\cla Y \to \cla X$ is classically of finite presentation, there exists an $\alpha$ such that the $\cla X$-morphism $Y_{1} \to \cla Y$ factors through a morphism $Y_{\alpha} \to \cla Y$, see \cite[Prop.~8.14.2]{EGAIV3}. 
Thus, so far, we have constructed a commutative diagram of classical schemes finitely presented over $\cla X$
\[
\begin{tikzcd}
 Y_{\alpha} \ar[d, hook]\ar[r] & \cla Y \ar[d] \\ 
 \cla {\widetilde X} \ar[r] & \cla X 
\end{tikzcd}
\]
in which all morphisms are isomorphisms over $\cla U$. The lower right corner is the classical truncation of the cospan $\widetilde X \xrightarrow{p} X \xleftarrow{f} Y$. 

As the composite of the affine open immersion $\cla U \to Y_{1}$ with the closed immersion $Y_{1} \hookrightarrow Y_{\alpha}$, the open immersion $\cla U \to Y_{\alpha}$ is affine, too. By 
Lemma~\ref{lem:nilOpenPushout} we may hence form the pushout $Y_{2} = Y_{\alpha} \sqcup_{\cla U} U$ of derived schemes, for which we have $\cla Y_{2} = Y_{\alpha}$. 
As $p$ and $f$ are isomorphisms over $U$, we get induced morphisms $h_{2}\colon Y_{2} \to \widetilde X$, $g_{2}\colon Y_{2} \to Y$, and a commutative diagram 
\[
\begin{tikzcd}
 Y_{2} \ar[d, hook, "h_{2}"']\ar[r, "g_{2}"] & Y \ar[d, "f"] \\ 
 \widetilde X \ar[r, "p"] & X. 
\end{tikzcd}
\]
Note that $h_{2}$ is a closed immersion, as this only depends on the underlying map of classical schemes. Moreover, all morphisms in the above diagram are $U$-modifications. 
However, $h_{2}$ and $g_{2}$ need not satisfy the desired finiteness condition. 
In order to remedy this, we consider the induced morphism $k_{2} = (h_{2}, g_{2})\colon Y_{2} \to \widetilde X \times_{X} Y$. As $\cla Y \to \cla X$ is separated, and as $h_{2}$ is a closed immersion, $k_{2}$ is a closed immersion, too. It is also an isomorphism over $U$.
By construction, the map of underlying classical schemes $\cla{k_{2}} \colon Y_{\alpha} = \cla Y_{2} \to \cla (\widetilde X \times_{X} Y)$ is classically of finite presentation. We can hence apply Lemma~\ref{lem:closed-subscheme} to obtain a factorization of $k_{2}$ through a closed derived subscheme $k'\colon Y' \hookrightarrow \widetilde X \times_{X} Y$ such that $k'$ is  locally of finite presentation and an isomorphism over $U$, and $\cla Y' \cong \cla Y_{2}$. 
Define $g$ and $h$ to be the composites of $k'$ with the two projections from $\widetilde X \times_{X} Y$ to $Y$ and $\widetilde X$, respectively. Both are $U$-modifications. As $p\colon \widetilde X \to X$ is locally of finite presentation, so is the first projection, and hence $g$. If $f$ is locally of finite presentation (or lafp, or locally of finite generation to order $n$), then so is the second projection and hence also $h$. 
As the underlying map of classical schemes $\cla h$ identifies with $\cla h_{2} \colon\cla Y_{2} \hookrightarrow \cla{\widetilde{X}}$, it is a closed immersion. 
This finishes the construction of the asserted commutative diagram, and the proof of the required finiteness conditions.

It remains to prove the claim about ample line bundles. 
As discussed in~\ref{subs:derived-blowups}, the canonical line bundle $\cO(1)$ on the derived blowup $\widetilde X$ is $p$-ample.
As $h$ is a closed immersion and thus in particular affine, the pullback $h^{*}\cO(1)$ is then $(p\circ h)$-ample.
\end{proof}

\section{Pro-cdh descent for connective localizing invariants}
	\label{sec:procdh}

In this section, we prove our  main results on pro-descent for localizing invariants. The strategy is the same as in \cite{KST}: We first prove the result for the special cases of derived blowups and finite modifications and then use the geometric input from Theorem~\ref{thm:structure-abu} to handle the general case.

We begin by fixing some notation. 
Let $k$ be a fixed commutative base ring (e.g.~$k=\Z$).
If $E$ is an additive invariant of small $k$-linear $\infty$-categories with values in a stable presentable $\infty$-category $\cC$, e.g.~the $\infty$-category of spectra, and $X$ is a qcqs derived $k$-scheme, we write $E(X)$ for $E(\Perf(X))$.
Let $Z \subseteq |X|$ be a closed subset with quasi-compact open complement. Recall from \ref{subs:formal-completion} that the formal completion $X_{Z}^{\wedge}$ is an ind-derived scheme. Applying $E$ we thus obtain a pro-object $E(X_{Z}^{\wedge})$. We  write $E(X, X_{Z}^{\wedge})$ for the relative term $\fib(E(X) \to E(X_{Z}^{\wedge}))$ in $\Pro(\cC)$.

A version of the following Proposition was first proven in \cite{KST}. There all schemes were assumed to be classical Noetherian schemes, $E$ was $K$-theory, and the result only gave a weakly cartesian square of pro-spectra, see below for this notion.
In a letter to Kerz, Antieau \cite{Antieau-notes} described a simplification of the proof which at the same time gives a cartesian square. We thank Ben Antieau for allowing us to include his argument in our paper. 

\begin{prop}\label{prop:descent-derived-bu}
Let $X$, $Z$, and $E$ be as above.
Let $\widetilde X \to X$ be a derived blowup in some quasi-smooth closed immersion $S \hookrightarrow X$ with $S$ set-theoretically contained in $Z$. Then the square 
\[
\begin{tikzcd}
 E(X) \ar[d]\ar[r] & E(X_{Z}^{\wedge}) \ar[d] \\ 
 E(\widetilde X) \ar[r] & E(\widetilde X_{Z}^{\wedge}) 
\end{tikzcd}
\]
is cartesian in $\Pro(\cC)$.
\end{prop}
\begin{proof}
Let $r\geq 1$ be the virtual codimension of the derived blowup, and let $D$ be its exceptional divisor, i.e., the universal virtual Cartier divisor on the derived blowup, so that there is a commutative diagram
\[
\begin{tikzcd}
 D \ar[d, "q"]\ar[r, hook, "j"] & \widetilde X \ar[d, "p"] \\ 
 S \ar[r, hook, "i"] & X .
\end{tikzcd}
\]
Recall from \cite[Thm.~C]{KhanSemiOrthogonal} that $\Perf(\widetilde X)$ has a semi-orthogonal decomposition as follows. The functor $p^{*}\colon \Perf(X) \to \Perf(\widetilde X)$ is fully faithful, denote its essential image by $\cB(0)$. For $1 \leq k \leq r-1$, the composed functor $j_{*}(q^{*}(-)\otimes_{\cO_{D}} \cO_{D}(-k)) \colon \Perf(S) \to \Perf(\widetilde X)$ is fully faithful, denote its essential image by $\cB(-k)$. Then the sequence of full subcategories $(\cB(0), \cB(-1), \dots, \cB(-r+1))$ forms a semi-orthogonal decomposition of $\Perf(\widetilde X)$. In particular, there is a decomposition
\begin{equation}\label{eq:semi1}
E(\widetilde X) \simeq E(\cB(0)) \oplus \bigoplus_{k=1}^{r-1} E(\cB(-k)) \simeq E(X) \oplus \bigoplus_{k=1}^{r-1} E(S).
\end{equation}

Now let $Z' \hookrightarrow X$ be any closed derived subscheme with $|Z'|=Z$. Note that all $\infty$-categories appearing above are in fact $\Perf(X)$-linear, as are the functors between them. In particular, we can base change the semi-orthogonal decomposition of $\Perf(\widetilde X)$ along $\Perf(X) \to \Perf(Z')$. As there are canonical equivalences (as follows from \cite[Cor. 9.4.3.8]{SAG} by passing to compact objects; see also \cite[Thm.~4.7]{BZFN} with slightly different hypotheses)
\begin{align*}
&\Perf(\widetilde X)\otimes_{\Perf(X)} \Perf(Z') \simeq \Perf(\widetilde X \times_{X} Z'), \\
&\Perf(S) \otimes_{\Perf(X)} \Perf(Z') \simeq \Perf(S \times_{X} Z'),
\end{align*}
we conclude that $\Perf(\widetilde X \times_{X} Z')$ admits a semi-orthogonal decomposition 
\[
(\cB(0)_{Z'}, \cB(-1)_{Z'}, \dots, \cB(-r+1)_{Z'})
\]
with $\cB(0)_{Z'} \simeq \Perf(Z')$ and $\cB(-k)_{Z'} \simeq  \Perf(S \times_{X} Z')$ for $1\leq k \leq r-1$. In particular,
\begin{equation}
	\label{eq:semi2}
E(\widetilde X \times_{X} Z') \simeq E(Z') \oplus \bigoplus_{k=1}^{r-1} E(S \times_{X} Z').
\end{equation}
Recall from \eqref{eq:completion} that $E(X_{Z}^{\wedge}) \in \Pro(\cC)$ is given concretely as the pro-object $\{ E(Z')\}_{Z' \hookrightarrow X, |Z'|=Z}$  where $Z'$ runs through all closed derived subschemes of $X$ with $|Z'|=Z$. As formal completion is compatible with base change (see~\ref{subs:formal-completion}), we similarly have $E(\widetilde X_{Z}^{\wedge}) = \{E(\widetilde X \times_{X} Z')\}_{Z' \hookrightarrow X, |Z'|=Z}$.
Comparing the decompositions \eqref{eq:semi1} and \eqref{eq:semi2} it thus suffices to prove that the functor induced by pullback
\[
\Perf(S) \to \{ \Perf(S\times_{X} Z') \}_{Z' \hookrightarrow X, |Z'|=Z}
\]
is an equivalence of pro-$\infty$-categories. For this, it suffices to check that the map of ind-derived schemes $\{S \times_{X} Z'\}_{Z' \hookrightarrow X, |Z'|=Z} \to S$ is an equivalence. But this is clear: By \ref{subs:formal-completion} again, the source represents the $Z$-completion $S_{Z}^{\wedge}$ of the target. As $S$ is set-theoretically contained in $Z$, we clearly have $S_{Z}^{\wedge} = S$.
\end{proof}

Let $\ell$ be an integer. Recall from \cite[Def.~2.5]{LT} that a spectra valued localizing invariant $E$ is called \emph{$\ell$-connective} if, for any $n$-connective map ($n\geq 1$)  of connective $\E_{1}$-ring spectra $A \to B$, the induced map $E(A) \to E(B)$ is $(n+\ell)$-connective. For example, $K$-theory, topological cyclic homology $\TC$ and rational negative cyclic homology $\HN(-\otimes\Q/\Q)$ are 1-connective, $\THH$ is 0-connective; see \cite[Ex.~2.6]{LT}.

Recall also that a map of pro-spectra $\{C_{\alpha}\}_{\alpha} \to \{D_{\alpha}\}_{\alpha}$ is called a \emph{weak equivalence} if each truncation $\{\tau_{\leq n}C_{\alpha}\}_{\alpha} \to \{ \tau_{\leq n}D_{\alpha}\}_{\alpha}$ is an equivalence in $\Pro(\Sp)$ and there are similar notions of being weakly cartesian, weakly contractible, and so on; see \cite[Def.~2.27]{LT}.

\begin{prop}\label{prop:descent-finite-modification} 
Let $E$ be a localizing invariant of small $k$-linear $\infty$-categories that is $\ell$-connective for some integer $\ell$.
Let $f\colon Y \to X$ be a finite, lafp morphism of qcqs derived $k$-schemes which is an isomorphism outside the closed subset $Z$ with $|X|\setminus Z$ quasi-compact.
Then the commutative square of pro-spectra
\begin{equation}\label{diag:1111}
\begin{tikzcd}
 E(X) \ar[d]\ar[r] & E(X_{Z}^{\wedge}) \ar[d] \\ 
 E(Y) \ar[r] & E(Y_{Z}^{\wedge}).
\end{tikzcd}
\end{equation}
is weakly cartesian.
\end{prop}

\begin{proof}
We first reduce to the case that $X$ is affine: As $X$ is qcqs, we can write $X$ as the colimit of a finite diagram of open affine subschemes $V_{i} \hookrightarrow X$. As any localizing invariant satisfies Zariski descent , we get $E(X) = \lim_{i} E(V_{i})$ and  $E(Y) = \lim_{i} E(Y \times_{X} V_{i})$.
If $Z' \hookrightarrow X$ is a closed subscheme with $|Z'|=Z$, we also have $E(Z') = \lim_{i} E(V_{i}\times_{X} Z')$.
As finite limits in $\Pro(\cC)$ are computed level-wise and formal completion is compatible with base change, this implies $E(X_{Z}^{\wedge}) = \lim_{i} E((V_{i})_{Z}^{\wedge})$ and similarly for $E(Y_{Z}^{\wedge})$. Replacing $X$ by $V_{i}$ and $Z$ by $V_{i} \cap Z$ we thus reduce to the case that $X$ is affine.

So assume now that $X$ is affine, say $X=\Spec(A)$. As $f$ is finite, also $Y$ is affine, say $Y=\Spec(B)$. Let $\phi\colon A \to B$ be the corresponding morphism of derived rings and write
 $J = \fib(A \to B)$. As $A$ and $B$ are connective, $J$ is $(-1)$-connective.
 By \cite[Cor.~5.2.2.2]{SAG} the $A$-algebra $B$ is almost perfect as an $A$-module, hence also $J$ is almost perfect as an $A$-module, i.e.~$\tau_{\leq n}J$ is a compact object in $\tau_{\leq n+1}\Mod(A)_{\geq -1}=\Mod(A)_{[-1,n]}$ for every $n$.\footnote{The $(n+1)$-truncated objects in $\Mod(A)_{\geq -1}$ are precisely the $n$-truncated, $(-1)$-connective objects in $\Mod(A)$ with respect to the standard t-structure. Hence the usual truncation $\tau_{\leq n}J$ coming from the t-structure is the categorical $(n+1)$-truncation $\Mod(A)_{\geq -1} \to \tau_{\leq n+1}\Mod(A)_{\geq -1}$.} 
 
Choose $f_{1}, \dots, f_{r} \in \pi_{0}(A)$ whose zero set is $Z$. Recall from \eqref{eq:completion-affine} that $X_{Z}^{\wedge}$ is then represented by the ind-derived scheme $\{ \Spec(A\mmod f_{1}^{\alpha}, \dots, f_{r}^{\alpha} ) \}_{\alpha\geq 1}$.
Consider the commutative diagram of derived rings
\begin{equation}
	\label{diag1}
\begin{tikzcd}
 A \ar[d]\ar[r] & A \mmod f_{1}^{\alpha}, \dots, f_{r}^{\alpha} \ar[d] \\ 
 B \ar[r] & B \mmod f_{1}^{\alpha}, \dots, f_{r}^{\alpha}.
\end{tikzcd}
\end{equation}
We claim that as pro-system in $\alpha$, this square is weakly cartesian. The map of vertical fibres (in $A$-modules) is the canonical map
\[
J \lto J \mmod f_{1}^{\alpha}, \dots, f_{r}^{\alpha},
\]
so we have to prove that this map is a weak equivalence as a pro-system in $\alpha$. For any $i = 1, \dots, r$, the fibre of the map of pro-systems $J \to \{J \mmod f_{i}^{\alpha}\}_{\alpha}$ is the pro-system 
\begin{equation}
	\label{eq:prosystem}
\{ \ J \stackrel{f_{i}}{\longleftarrow} J \stackrel{f_{i}}{\longleftarrow}\dots \ \}.
\end{equation}
We show below that this system is weakly contractible. We then get weak equivalences $J \xrightarrow{\simeq} \{J \mmod f_{1}^{\alpha}\}_{\alpha}$ and, taking derived quotients by powers of $f_{2}$,   $\{J \mmod f_{2}^{\alpha}\}_{\alpha} \xrightarrow{\simeq} \{J\mmod f_{1}^{\alpha}, f_{2}^{\alpha}\}_{\alpha}$. Composing with $J \xrightarrow{\simeq} \{J \mmod f_{2}^{\alpha}\}_{\alpha}$ we get the weak equivalence $J \xrightarrow{\simeq} \{J \mmod f_{1}^{\alpha}, f_{2}^{\alpha}\}_{\alpha}$ and continuing like this we finally arrive at the weak equivalence  $J \xrightarrow{\simeq} \{J \mmod f_{1}^{\alpha},\dots,  f_{r}^{\alpha}\}_{\alpha}$. So $\{\eqref{diag1}\}_{\alpha}$ is indeed weakly cartesian.

The assumption that $f$ be an isomorphism outside $Z$ implies that $J[f_{i}^{-1}] = 0$ for $i = 1, \dots, r$.
Note that 
\[
J[f_{i}^{-1}] = \colim( J \xrightarrow{f_{i}} J \xrightarrow{f_{i}} J \xrightarrow{f_{i}} \dots).
\]
As the standard t-structure on $\Mod(A)$ is compatible with filtered colimits, we have
\[
0 = \tau_{\leq n} J[f_{i}^{-1}] = \colim( \tau_{\leq n} J \xrightarrow{f_{i}} \tau_{\leq n} J \xrightarrow{f_{i}} \tau_{\leq n}  J \xrightarrow{f_{i}} \dots)
\]
As $\tau_{\leq n}J$ is compact in $\Mod(A)_{[-1,n]}$, we have
\[
0 = \pi_{0}(\map(\tau_{\leq n}J, \tau_{\leq n}J[f_{i}^{-1}]) ) \cong \colim \pi_{0}( \map(\tau_{\leq n}J, \tau_{\leq n}J))
\]
which means that there is an $N$ such that the power $f_{i}^{N}$ acts nullhomotopically on $\tau_{\leq n} J$. It follows that \eqref{eq:prosystem} is weakly contractible, and hence the pro-system of squares $\{\text{\eqref{diag1}}\}_{\alpha}$ is indeed weakly cartesian. 

Note that for every $\alpha$ the canonical map $B \otimes_{A} (A \mmod f_{1}^{\alpha}, \dots, f_{r}^{\alpha}) \to B \mmod f_{1}^{\alpha}, \dots, f_{r}^{\alpha}$ is an equivalence. Thus  we may apply the variant of \cite[Thm.~2.32]{LT} for $\ell$-connective localizing invariants to deduce that $\{\text{\eqref{diag1}}\}_{\alpha}$ induces a weakly cartesian square of $E$-theory pro-spectra.

\end{proof}

\begin{rem}
	\label{rem:pro-descent-for-class-fp}
 In Proposition~\ref{prop:descent-finite-modification} one can actually relax the finiteness assumption if one adds other hypotheses: As the proof shows, we only need that the pro-systems \eqref{eq:prosystem} are weakly contractible for each $i$ (using notation of the proof). This is satisfied if, on each truncation $\tau_{\leq n}J$, some power of each $f_{i}$ acts null-homotopically.

If $X$ and $Y$ are $n$-truncated, then also $J$ is $n$-truncated. It is then enough to assume that $J$ is perfect to order $n$ in the sense of \cite[Def.~2.7.0.1]{SAG} in order to conclude that the pro-systems \eqref{eq:prosystem} are weakly contractible.

As a special case, if the map $f$ in Proposition~\ref{prop:descent-finite-modification} is a closed immersion of classical qcqs schemes which is classically finitely presented, then the conclusion of the proposition holds, i.e.~\eqref{diag:1111} is weakly cartesian. On the other hand, we cannot drop the finite presentation assumption, as the following example shows.
\end{rem}

\begin{ex}
Let $A$ be a discrete, commutative ring containing an element $f$ and a non-trivial ideal $J$ such that $J[f^{-1}]=0$ and $J$ is $f$-divisible. Note that this implies that $J^{2} =0$. For example, take $A = \Z \oplus \Q_{p}/\Z_{p}$, $f=p$, $J = \Q_{p}/\Z_{p}$.

The closed immersion $Y = \Spec(A/J) \to X= \Spec(A)$ is then an isomorphism outside $Z = V(f)$. We claim that the square
\begin{equation}
	\label{diag:non-fp-procdh}
	\begin{tikzcd}
	K(X) \ar[d]\ar[r] & K(X_{Z}^{\wedge})  \ar[d] \\ 
	K(Y) \ar[r] &  K(Y_{Z}^{\wedge})
	\end{tikzcd}
\end{equation}
is not weakly cartesian. 

Aiming at a contradiction, assume \eqref{diag:non-fp-procdh} is weakly cartesian.
Fix some positive integer $n$. As $J$ is $f$-divisible, 
\[
J \mmod f^{n} = \fib(A\mmod f^{n} \to (A/J)\mmod f^{n})
\]
is 1-connective. As a consequence, the map of $K$-theory spectra
\[
K(A\mmod f^{n}) \to K((A/J)\mmod f^{n})
\]
is 2-connective. So also the right vertical map, and hence the left vertical map in \eqref{diag:non-fp-procdh} is 2-connective. In particular, the map $K_{1}(A) \to K_{1}(A/J)$ is an isomorphism. However, its retract $A^{*} \to (A/J)^{*}$ is not injective as every $1+x$ with $x \in J$ is an element of its kernel.
\end{ex}

\begin{rem}
 There is a version of the above Proposition for stacks: Let $X$ be a qcqs ANS derived algebraic stack  \cite[A.1]{BachmannCatMilnor}, $Y \to X$ a finite, locally almost finitely presented morphism of derived algebraic stacks, and $Z \hookrightarrow X$ a closed immersion with quasi-compact open complement. Let $E$ be any connected localizing invariant in the sense of \cite[Def.~C.1.3]{BachmannCatMilnor} (e.g.~a 2-connective or a finitary 1-connective localizing \cite[Rem.~C.1.5]{BachmannCatMilnor}). Then the square  \eqref{diag:1111} of pro-spectra is weakly cartesian.
	
	Indeed, the proof of \cite[Thm.~4.2.1]{BachmannCatMilnor} works with the following changes:  As in the proof of Lemma~2.3.2 in \opcit, the proof of our Proposition~\ref{prop:descent-finite-modification} shows that the formally completed square
	\begin{equation} \label{diag:stacks}
	\begin{tikzcd}
	 Y_{Z}^{\wedge} \ar[d]\ar[r] & Y \ar[d] \\ 
	 X_{Z}^{\wedge} \ar[r] & X 
	\end{tikzcd}
	\end{equation}
	is weakly cocartesian. As in the proof of Theorem~2.4.1 of \opcit~this implies that the square of derived (pro-)categories induced by \eqref{diag:stacks} is weak pro-Milnor and satisfies weak pro-base change in the sense of Definitions~C.2.4 and C.2.6 there. Hence by Theorem~C.3.1 there the square \eqref{diag:1111} is weakly cartesian.
\end{rem}

We now come to our main descent result, which in particular includes Theorem~\ref{thmA}.  

\begin{thm}\label{thm:pro-cdh}
Let $E$ be a localizing invariant of small $k$-linear $\infty$-categories  which is $\ell$-connective for some integer $\ell$.
Let $X$ be a qcqs derived $k$-scheme, $U \subseteq X$ a quasi-compact open subscheme, and denote by $Z$ the closed subset $X \setminus U$.
Let $f\colon Y \to X$ be a  locally almost finitely presented $U$-modification  of $X$. Then the square of pro-spectra
\begin{equation}\label{diag:2222}
\begin{tikzcd}
 E(X) \ar[d]\ar[r] & E(X_{Z}^{\wedge}) \ar[d] \\ 
 E(Y) \ar[r] & E(Y_{Z}^{\wedge})
\end{tikzcd}
\end{equation}
is weakly cartesian.
\end{thm}

If $X$ and $Y$ are Noetherian classical schemes, and $f$ is classically of finite type, then $f$ is lafp (see \ref{subs:finiteness-conditions}). Moreover, the formal derived completions are equivalent to the classical formal completions (see \ref{subs:formal-completion}). We thus recover the classical pro-cdh descent statement as formulated (for $K$-theory) for example in \cite[Thm.~A]{KST}.

\begin{proof}[Proof of Theorem~\ref{thm:pro-cdh}]
We first prove the theorem under the additional assumption that $Y$ carries an $f$-ample line bundle.
Exactly as in the proof of Proposition~\ref{prop:descent-finite-modification} we reduce to the case that $X$ is affine.
Then $Y$ carries an ample line bundle.

By Theorem~\ref{thm:structure-abu}, we find an lafp $U$-modification $g\colon Y' \to Y$ such that $f\circ g$ factors as an lafp, closed $U$-modification $h\colon Y' \hookrightarrow \widetilde X$ followed by a derived blowup $p\colon \widetilde X \to X$ with center set-theoretically contained in $Z$. Both of these are isomorphisms over $U$. By Proposition~\ref{prop:descent-derived-bu} and Proposition~\ref{prop:descent-finite-modification}, the maps of relative $E$-theory pro-spectra
 \[
 E(X,X_{Z}^{\wedge}) \to E(\widetilde X, \widetilde X_{Z}^{\wedge}) \to E(Y', {Y'}_{Z}^{\wedge})
 \]
are weak equivalences. Hence also the composite
\begin{equation}\label{eq:composite}
E(X,X_{Z}^{\wedge}) \to E(Y, Y_{Z}^{\wedge}) \to E(Y', {Y'}_{Z}^{\wedge})
\end{equation}
is a weak equivalence. It follows that $E(X,X_{Z}^{\wedge}) \to E(Y, Y_{Z}^{\wedge})$ is the inclusion of a direct summand (in the weak sense). As $Y$ carries an ample line bundle, we can repeat this argument for the $U$-modification $Y' \to Y$. So also the second map in \eqref{eq:composite} is the inclusion of a direct summand. It now follows that both maps are in fact weak equivalences.

We now prove the theorem for general $Y$.
As before, we may assume that $X$ is affine. By Theorem~\ref{thm:structure-abu} again, there exists an lafp $U$-modification $g \colon Y' \to Y$ such that $Y'$ carries an ample line bundle relative to $X$. As $X$ is affine, this line bundle is in fact ample and it is also ample relative to $Y$. Hence, by Step 1, the maps of pro-spectra $E(X, X_{Z}^{\wedge}) \to E(Y', {Y'}_{Z}^{\wedge})$ and $E(Y, Y_{Z}^{\wedge}) \to E(Y', {Y'}_{Z}^{\wedge})$ both are weak equivalences. By 3-for-2, also $E(X, X_{Z}^{\wedge}) \to E(Y, Y_{Z}^{\wedge})$ is a weak equivalence, as desired.
\end{proof}

\section{Pro-cdh descent for the cotangent complex and motivic cohomology}

We fix some base ring $k$ (e.g., $k = \Z$).
For $X$ a qcqs derived $k$-scheme, we denote by $L_{X} \in \QCoh(X)$ its (algebraic) cotangent complex relative to $k$, and by $L^{i}_{X}$ its $i$-th derived exterior power ($i\geq 0$). If $X =\Spec(A)$ is affine, we also write $L^{i}_{A}$ for the $A$-module $\Gamma(X, L^{i}_{X})$ corresponding to $L^{i}_{X}$. If $Y \to X$ is a morphism of  qcqs derived $k$-schemes, we denote by $L_{Y/X}$ its relative cotangent complex and by $L^{i}_{Y/X}$ its derived exterior powers. Similarly for a morphism of derived $k$-algebras $A \to B$.

If $Z \subseteq |X|$ is a closed subset with quasi-compact open complement, we obtain a \emph{pro-completion along $Z$} functor
\[
(-)_{Z}^{\wedge} \colon \QCoh(X) \to \Pro(\QCoh(X))
\]
by pulling back to the ind-derived scheme $X_{Z}^{\wedge}$ and pushing forward. The usual formal completion along $Z$ is given by composing the above functor with $\lim\colon \Pro(\QCoh(X)) \to \QCoh(X)$.

\begin{lemma}
	\label{lem:pro-completion-aperf}
Let $X = \Spec(A)$ be an affine derived $k$-scheme, $Z \subseteq |X|$ a closed subset with quasi-compact open complement. Let $M \in \Mod(A \on Z)^{\aperf}$ be an almost perfect $A$-module supported on $Z$. Then the canonical map
\[
M \to M_{Z}^{\wedge}
\]
is a weak equivalence in $\Pro(\Mod(A))$.
\end{lemma}

\begin{proof}
This was proven in the proof of Proposition~\ref{prop:descent-finite-modification} (replace $J$ there by $M$).
\end{proof}

For a pro-system of derived rings $\{ A(\alpha)\}_{\alpha}$ we denote by $\Pro(\Mod)(\{A(\alpha)\})$ the $\infty$-category of pro-systems of modules over the pro-ring $\{A(\alpha)\}_{\alpha}$ (see \cite[\S 2.4]{LT} for a precise definition).

\begin{lemma}
	\label{lem:relative-pro-cotangent}
In the situation of the previous lemma, choose $f_{1}, \dots, f_{r} \in \pi_{0}(A)$ defining $Z$, and write $A(\alpha) = A \mmod f_{1}^{\alpha}, \dots, f_{r}^{\alpha}$. Then the pro-system
\[
\{ L_{A(\alpha)/A} \}_{\alpha} \in \Pro(\Mod)(\{A(\alpha)\})
\]
vanishes. In fact, all transition maps in this pro-system are null-homotopic.
\end{lemma}
\begin{proof}
This follows by base change from the universal case: Let $R = k[T_{1}, \dots, T_{r}]$ be the polynomial ring over $k$, let $R \to k$ be the map sending all $T_{i}$ to $0$, and let $g_{\alpha}\colon R \to A$ be the map sending $T_{i}$ to $f_{i}^{\alpha}$. Then $A(\alpha) = k \otimes_{R, g_{\alpha}} A$ and consequently
\[
L_{A(\alpha)/A} \simeq L_{k/R} \otimes_{R, g_{\alpha}} A \simeq L_{k/R} \otimes_{k} A(\alpha).
\]
Consider the pro-system $\{L_{k/R}\}_{\alpha}\in \Pro(\Mod(k))$ whose transition maps are induced by the maps $R \to R$ sending the $T_{i}$ to $T_{i}^{\beta}$. Then 
\[
\{ L_{A(\alpha)/A} \}_{\alpha} \simeq \{L_{k/R}\}_{\alpha} \otimes_{k} \{A(\alpha)\}_{\alpha}.
\]
Hence it suffices to prove that all transition maps $L_{k/R} \to L_{k/R}$ are null-homotopic in $\Mod(k)$.
As $R \to k$ has a section, the transitivity triangle yields an equivalence $L_{k/R} \simeq \Sigma L_{R/k} \otimes_{R} k$.
As $R$ is a polynomial ring, $L_{R/k}$ is discrete, given by the module of Kähler differentials $\Omega^{1}_{R/k} = \bigoplus_{i} RdT_{i}$.  The map  $T_{i} \mapsto T_{i}^{\beta}$ induces $dT_{i} \mapsto \beta T_{i}^{\beta-1}dT_{i}$ in $\Omega^{1}_{R/k}$, and hence the zero map in $\Omega^{1}_{R/k} \otimes_{R} k \simeq \Sigma^{-1}L_{k/R}$ for $\beta>1$.  This proves our claim.
\end{proof}

\begin{lemma}
	\label{lem:absolute-cotangent-completion}
In the situation of Lemma~\ref{lem:relative-pro-cotangent}, the canonical map
\[
\{ L^{i}_{A} \otimes_{A} A(\alpha)\}_{\alpha} \to \{ L^{i}_{A(\alpha)}\}_{\alpha}
\]
is an equivalence in $\Pro(\Mod)(\{A(\alpha)\})$.
\end{lemma}

\begin{proof}
The transitivity triangle for the maps $A \to A(\alpha)$ and Lemma~\ref{lem:relative-pro-cotangent} imply the case $i=1$.  Passing to derived exterior powers over $\{A(\alpha)\}$ implies the general case.
\end{proof}

The following theorem was suggested by Matthew Morrow. It generalizes \cite[Thm.~2.4]{Mor} (see also \cite[Lemma~8.5]{ElmantoMorrow}).
\begin{thm}
	\label{thm:procdh-cotangent}
Let $X$ be a qcqs derived $k$-scheme, $U \subseteq X$ a quasi-compact open subscheme, and denote by $Z$ the closed subset $X \setminus U$.
Let $f\colon Y \to X$ be a  locally almost finitely presented $U$-modification  of $X$. Then for every $i\geq 0$, the square of pro-spectra (or pro-complexes)
\begin{equation}\label{diag:pro-cdh-L}
\begin{tikzcd}
 \Gamma(X, L^{i}_{X}) \ar[d]\ar[r] & \Gamma(X_{Z}^{\wedge}, L^{i}_{X_{Z}^{\wedge}}) \ar[d] \\ 
 \Gamma(Y, L^{i}_{Y}) \ar[r] & \Gamma(Y_{Z}^{\wedge}, L^{i}_{Y_{Z}^{\wedge}})
\end{tikzcd}
\end{equation}
is weakly cartesian.
\end{thm}

Here, similarly as in the previous section,  $\Gamma(X_{Z}^{\wedge}, L^{i}_{X_{Z}^{\wedge}})$ denotes the pro-object $\{ \Gamma(Z', L^{i}_{Z'}) \}_{Z' \hookrightarrow X}$
where $Z' \hookrightarrow X$ runs through all closed immersions of derived schemes with $|Z'| = Z$.

\begin{proof}
By Zariski descent we may assume that $X= \Spec(A)$ is affine. Choose $f_{1}, \dots, f_{r} \in \pi_{0}(A)$ defining $Z$, and write $A(\alpha) = A \mmod f_{1}^{\alpha}, \dots, f_{r}^{\alpha}$ so that $X_{Z}^{\wedge} = \{\Spec(A(\alpha))\}_{\alpha\geq 1}$. Abusing notation slightly, we also write $f_{*}L^{i}_{Y} \in \Mod(A)$ for the module $\Gamma(X, f_{*}L^{i}_{Y}) = \Gamma(Y, L^{i}_{Y})$ corresponding to $f_{*}L^{i}_{Y} \in \QCoh(X)$.

By Lemma~\ref{lem:absolute-cotangent-completion} we  have
\[
\Gamma(X_{Z}^{\wedge}, L^{i}_{X_{Z}^{\wedge}}) \simeq L^{i}_{A} \otimes_{A} \{ A(\alpha)\}.
\]
Similarly, using  affine coverings of $Y$ and the usual induction we get
\[
\Gamma(Y_{Z}^{\wedge}, L^{i}_{Y_{Z}^{\wedge}}) \simeq (f_{*}L^{i}_{Y}) \otimes_{A} \{A(\alpha)\}.
\]
So we have to prove that
\begin{equation}\label{eq:cot1}
\begin{tikzcd}
 L^{i}_{A} \ar[d]\ar[r] & L^{i}_{A} \otimes_{A} \{ A(\alpha)\} \ar[d] \\ 
 f_{*}L^{i}_{Y} \ar[r] & (f_{*}L^{i}_{Y}) \otimes_{A} \{A(\alpha)\} 
\end{tikzcd}
\end{equation}
is weakly cartesian. The  map $L^{i}_{A} \to f_{*}L^{i}_{Y}$ factors as $L^{i}_{A} \to f_{*}f^{*}L^{i}_{A} \to f_{*}L^{i}_{Y}$, so it suffices to show that the two squares
\begin{equation}\label{eq:cot2}
\begin{tikzcd}
 L^{i}_{A} \ar[d]\ar[r] & L^{i}_{A} \otimes_{A} \{ A(\alpha)\} \ar[d] \\ 
 f_{*}f^{*}L^{i}_{A} \ar[r] & (f_{*}f^{*}L^{i}_{A}) \otimes_{A} \{A(\alpha)\} 
\end{tikzcd}
\text{ and }
\begin{tikzcd}
 f_{*}f^{*}L^{i}_{A} \ar[d]\ar[r] & (f_{*}f^{*}L^{i}_{A}) \otimes_{A} \{ A(\alpha)\} \ar[d] \\ 
 f_{*}L^{i}_{Y} \ar[r] & (f_{*}L^{i}_{Y}) \otimes_{A} \{A(\alpha)\} 
\end{tikzcd}
\end{equation}
are weakly cartesian. 

We first treat the left one. Using the projection formula, we rewrite that square as the tensor product of $L^{i}_{A}$ with the square
\begin{equation}\label{eq:cot3}
\begin{tikzcd}
 A \ar[d]\ar[r] & \{A(\alpha)\} \ar[d] \\ 
 f_{*}\cO_{Y} \ar[r] & f_{*}\cO_{Y} \otimes_{A} \{A(\alpha)\}. 
\end{tikzcd}
\end{equation} 
As $f$ is lafp, $f_{*}\cO_{Y}$ is almost perfect \cite[Thm.~5.6.0.2]{SAG}. On the other hand, by assumption $A \to f_{*}\cO_{Y}$ is an isomorphism outside $Z$. Hence the left vertical fibre in \eqref{eq:cot3} lies in $\Mod(A \on Z)^{\aperf}$. It now follows from Lemma~\ref{lem:pro-completion-aperf} that the map on vertical fibres is a weak equivalence, and hence that $\eqref{eq:cot3}$ is weakly cartesian. As $L^{i}_{A}$ is connective, tensoring with $L^{i}_{A}$ preserves weak equivalences and weakly cartesian squares (see e.g.~\cite[Lemma~2.29]{LT}). Thus the left-hand square in \eqref{eq:cot2} is weakly cartesian. 

We now consider the right square in \eqref{eq:cot2}. The transitivity triangle for $f$ gives rise to a finite filtration on $L^{i}_{Y}$ whose graded pieces are given by 
\[
f^{*}L^{k}_{X} \otimes_{\cO_{Y}} L^{i-k}_{Y/X}, \quad k=0, \dots, i.
\]
It follows that the left vertical cofibre in our square has a finite filtration whose graded pieces are given by
\[
f_{*}(f^{*}L^{k}_{X} \otimes_{\cO_{Y}} L^{i-k}_{Y/X}) \simeq L^{k}_{X} \otimes_{A} f_{*}L^{i-k}_{Y/X} , \quad k=0, \dots, i-1.
\]
As $f$ is lafp, the relative cotangent complex $L_{Y/X}$ is almost perfect \cite[Prop.~3.2.14]{LurieThesis} and so are its wedge powers $L^{i-k}_{Y/X}$. Again because $f$ is lafp, the direct images $f_{*}L^{i-k}_{Y/X}$ are almost perfect \cite[Thm.~5.6.0.2]{SAG}. As $f$ is an isomorphism outside $Z$, the sheaf $f_{*}L^{i-k}_{Y/X}$ is supported on $Z$ for $k<i$. So we may again apply Lemma~\ref{lem:pro-completion-aperf} to deduce that the map $f_{*}L^{i-k}_{Y/X} \to (f_{*}L^{i-k}_{Y/X}) \otimes_{A} \{A(\alpha)\}_{\alpha}$ is a weak equivalence for $k=0, \dots, i-1$ and hence so is the map 
\[
L^{k}_{X} \otimes_{A} f_{*}L^{i-k}_{Y/X} \to  (L^{k}_{X} \otimes_{A} f_{*}L^{i-k}_{Y/X}) \otimes_{A} \{A(\alpha)\}_{\alpha}.
\]
Using the above filtration, it follows that the map on vertical fibres in the right-hand square in \eqref{eq:cot3} is a weak equivalence, and hence that square is also weakly cartesian. This finishes the proof of the theorem.
\end{proof}

Combining Theorem~\ref{thm:procdh-cotangent} with the arguments of \cite[\S 8.1]{ElmantoMorrow} we obtain pro-cdh descent for Elmanto--Morrow's motivic cohomology denoted by $\Z(j)^{\mot}(-)$.

\begin{cor}
	\label{cor:pro-cdh-motivic}
Let $\F$ be a prime field.  
Let $X$ be a qcqs derived $\F$-scheme, $U \subseteq X$ a quasi-compact open subscheme, and denote by $Z$ the closed subset $X \setminus U$.
Let $f\colon Y \to X$ be a  locally almost finitely presented $U$-modification  of $X$. Then the square of pro-complexes
\[
\begin{tikzcd}
 \Z(j)^{\mot}(X) \ar[d]\ar[r] &  \Z(j)^{\mot}(X_{Z}^{\wedge}) \ar[d] \\ 
  \Z(j)^{\mot}(Y) \ar[r] &  \Z(j)^{\mot}(Y_{Z}^{\wedge})
\end{tikzcd}
\]
is weakly cartesian.
\end{cor}
\begin{proof}
The proof of Elmanto and Morrow goes through verbatim once we replace their Lemma~8.5 by the above Theorem~\ref{thm:procdh-cotangent}.
\end{proof}

\section{Generalized Weibel vanishing}
	\label{sec:Weibel-vanishing}

In this section we prove our generalized Weibel vanishing result. As its formulation involves the valuative dimension, we start by recalling the latter (Subsection~\ref{subsec:valdim}). In Subsection~\ref{subsec:reductions} we generalize some results about reductions of ideals  from \cite{SwansonHuneke} to non-Noetherian rings. These will go into the proof of Weibel vanishing in Subsection~\ref{subsec:Weibel-vanishing}.

\subsection{Valuative dimension}
	\label{subsec:valdim}

The valuative dimension of a commutative ring was introduced and studied by Jaffard~\cite{Jaffard}; we refer to \cite[\S 2.3]{EHIK} for an account and to \cite[Lemma~7.2]{Kelly:2024aa} which reproves some of Jaffard's key results in  modern language. For an integral domain $A$, it is defined as
\[
\vdim(A) = \sup \{ n \,|\, \exists A \subseteq V_{0} \subsetneq V_{1} \subsetneq \dots \subsetneq V_{n} \subseteq \Frac(A), \text{ $V_{i}$ valuation ring}\},
\]
and in general as $\vdim(A) = \sup \{ \vdim(A/\fp) \,|\, \fp \in \Spec(A) \}$.
For a classical scheme $X$, one then sets 
\[
\vdim(X) = \sup \{ \vdim(\cO_{X}(U)) \,|\, U \subseteq X \text{ affine open}\}.
\]
We remark that that $\vdim(X) = \sup \dim \operatorname{RZ}(X_\lambda)$ is the supremum of the Krull dimensions of the Riemann--Zariski spaces as $X_\lambda$ ranges over the irreducible components of $X$. That is, the supremum of the Krull dimensions of all blowups (of finite type) of all $X_\lambda$.
We always have $\dim(X) \leq \vdim(X)$ and equality holds if $X$ is Noetherian \cite[Ch.~IV, Thm.~1, Cor.~2 of Thm.~5]{Jaffard}.
Note that the valuative dimension only depends on the underlying reduced scheme of a scheme. In particular, it makes sense to talk about the valuative dimension of a closed subset of a scheme, by equipping it with any closed subscheme structure.

Let now $X$ be a derived or spectral scheme. The Krull dimension $\dim(X)$ of $X$ is defined to be the Krull dimension of the underlying topological space of $X$. The valuative dimension $\vdim(X)$ is defined to be the valuative dimension of the underlying classical scheme $\cla X$.

We will need the following result.
\begin{lemma}\label{lem;dimv} 
For a scheme $X$ and a point $x\in X$ with  closure $\ol{\{x\}}$ in $X$, 
\[
\vdim(\cO_{X,x}) +  \vdim(\ol{\{x\}}) \leq \vdim(X) .
\]
In particular,
\[
\vdim(\cO_{X,x}) +  \dim(\ol{\{x\}}) \leq \vdim(X) .
\]
 \end{lemma}
\begin{proof}
By definition it suffices to prove the lemma in case $X=\Spec(A)$ with $A$ integral. In this case, the first assertion  is \cite[Ch.~IV, Prop.~2]{Jaffard}; indeed, this follows from the basic facts that for any field extension $L/K$ any valuation $R \subseteq K$ admits at least one extension $S \subseteq L$, and $\Spec(S) \to \Spec(R)$ is always surjective. The second assertion follows from the fact that $\dim(Y) \leq \vdim(Y)$ for every scheme $Y$.
\end{proof}

\subsection{Reductions}
	\label{subsec:reductions}
	
Let $R$ be a ring, and let $J \subseteq I$ be ideals in $R$. Then $J$ is called a reduction of $I$ if there is a positive integer $n$ such that $JI^{n} = I^{n+1}$. Associated with an ideal $I$ we have its Rees algebra
\[
R[It] = \sum_{n\geq 0} I^{n}t^{n} \subseteq R[t], \quad R[It] \cong \bigoplus_{n\geq 0} I^{n}
\]
and its extended Rees algebra
\[
R[It, t^{-1}] = \sum_{n<0} Rt^{n} + \sum_{n\geq 0} I^{n}t^{n} \subseteq R[t, t^{-1}].
\]
So $\operatorname{Proj}(R[It])$ is the blowup of $\Spec(R)$ in $\Spec(R/I)$, and $\Spec(R[It, t^{-1}])$ is a flat deformation of $\Spec(R)$ to the normal cone $\Spec(\bigoplus_{n\geq 0} I^{n}/I^{n+1})$.

\begin{lemma} 
	\label{lem:reduction-finiteness}
Assume that $J \subseteq I$ are  ideals in $R$. 
\begin{enumerate}
\item If $I$ is finitely generated, and $J$ is a reduction of $I$, then the ring extension $R[Jt] \subseteq R[It]$ is finite.
\item Conversely, if the ring extension $R[Jt] \subseteq R[It]$ is finite, then $J$ is a reduction of $I$.
\end{enumerate}
\end{lemma}
\begin{proof}
This is \cite[Th.~8.2.1]{SwansonHuneke} which is formulated for ideals in Noetherian rings. The proof given there works verbatim in the asserted generality.
\end{proof}

Let $R$ be a local ring with maximal ideal $\fm$ and residue field $\kappa$. Let $I \subset R$ be a finitely generated, proper ideal. Then the fibre cone
\[
 \cF_{I}(R) = R[It]\otimes^{\heartsuit}_{R} \kappa = \kappa \oplus I/\fm I \oplus I^{2}/\fm I^{2} \oplus \dots
\]
is a finitely generated $\kappa$-algebra and in particular Noetherian. Its Krull dimension $\ell(I) := \dim \cF_{I}(R)$ is called the analytic spread of $I$.

\begin{prop}
	\label{prop:analytic-spread-vdim}
Let $R$ be a local ring, $I \subseteq R$ a finitely generated, proper ideal. Then the analytic spread of $I$ is bounded by the valuative dimension of $R$:
\[
\ell(I) \leq \vdim(R).
\]
\end{prop}

\begin{proof}
This proof follows \cite[Prop.~5.1.6]{SwansonHuneke}, which treats the Noetherian case. 
We may assume that $\vdim(R)$ is finite.
Let $\gr_{I}(R) = R[It]\otimes^{\heartsuit}_{R} R/I = R/I \oplus I/I^{2} \oplus \dots = R[It,t^{-1}]/(t^{-1})$. As $\cF_{I}(R)$ is a quotient of $\gr_{I}(R)$ and $\cF_{I}(R)$ is Noetherian, we have
\[
\ell(I) = \vdim \cF_{I}(R) \leq \vdim \gr_{I}(R).
\]
As $t^{-1}$ is a non-zero divisor in $R[It, t^{-1}]$, it is not contained in any minimal prime ideal of $R[It, t^{-1}]$. Thus $\vdim \gr_{I}(R) \leq \vdim R[It, t^{-1}] - 1$ by \cite[Prop.~2.3.2(4)]{EHIK}.
We claim that $\vdim R[It, t^{-1}] = \vdim R +1$. This will finish the proof. To prove the claim, as the minimal prime ideals of $R[It, t^{-1}]$ are in bijection with those of $R$ (see the general discussion in \cite{SwansonHuneke} before Thm.~5.1.4), we may assume that $R$ is a domain. As $\Frac(R[It, t^{-1}]) = \Frac(R)(t)$ has transcendence degree 1 over $\Frac(R)$, \cite[Ch.~I, Thm.~2, p.~10]{Jaffard} implies that if $v'$ is any valuation of $\Frac(R[It, t^{-1}])$ extending a valuation $v$ of $\Frac(R)$ then the rank $\rank(v')$  of $v'$ is at most $\rank(v) + 1$. It follows that $\vdim R[It, t^{-1}] \leq \vdim R + 1$. 
The kernel of the surjection $R[It, t^{-1}] \to R, t \mapsto 1$, is a non-trivial prime ideal, hence $\vdim R[It, t^{-1}] > \vdim R$, for example by Lemma~\ref{lem;dimv}.
\end{proof}

\begin{prop}
	\label{prop:existence-reduction}
Let $R$ be a local ring and $I \subseteq R$ be a finitely generated, proper ideal with analytic spread $\ell(I)$. Then there exists an $n\geq 1$ and a reduction $J$ of $I^{n}$ which is generated by $\ell(I)$ elements.
\end{prop}

\begin{proof}
We follow the proof of \cite[Prop.~8.3.8]{SwansonHuneke}. Let $\fm$ denote the maximal ideal of $R$, $\kappa = R/\fm$.
By graded Noether normalization (see e.g.~\cite[Thm.~4.2.3]{SwansonHuneke}) there exists an integer $n$ and homogeneous elements $\bar a_{1}, \dots, \bar a_{\ell} \in \cF_{I}(R)_{n} = I^{n}/\fm I^{n}$ such that $A = \kappa[\bar a_{1}, \dots, \bar a_{\ell}] \subseteq \cF_{I}(R)$ is a polynomial ring and the extension $A \subseteq \cF_{I}(R)$ is finite. In particular, $\ell =\dim \cF_{I}(R) = \ell(I)$. 
Moreover, also the subextension $A \subseteq \bigoplus_{k\geq 0} I^{nk}/\fm I^{nk}$ is then finite (up to grading,  the latter algebra is $\cF_{I^{n}}(R)$). Let $nd$ be the maximal degree of a finite set of homogeneous generators. Then 
$I^{n(d+1)}/\fm I^{n(d+1)} = \sum_{i=1}^{\ell} \bar a_{i} \cdot I^{nd}/\fm I^{nd}$.
Choose lifts $a_{i} \in I^{n}$ of the $\bar a_{i}$ and let $J = (a_{1}, \dots, a_{\ell}) \subseteq I^{n}$. The previous equality implies $J I^{nd} + \fm I^{n(d+1)} = I^{n(d+1)}$. As $I$ is finitely generated, the Nakayama lemma gives $JI^{nd} = I^{n(d+1)}$, i.e.~$J$ is a reduction of $I^{n}$.
\end{proof}

\subsection{Weibel vanishing}
	\label{subsec:Weibel-vanishing}

The following theorem generalizes \cite[Thm.~B]{KST}, which treats the case of Noetherian classical schemes.

\begin{thm}\label{thm:Weibel-vanishing}
Let $X$ be a qcqs spectral scheme. 
Then the following hold.
\begin{enumerate}
\item $K_{-i}(X) = 0$ for all $i> \vdim(X)$.
\item For all $i \geq \vdim(X)$ and any integer $r \geq 0$, the pullback map $K_{-i}(X) \to K_{-i}(\A^{r}_{X})$ is an isomorphism.
\end{enumerate}
\end{thm}
In the spectral setting, the affine space in (2) is either the flat one or the smooth one, the assertion holds for both. In fact, one reduces immediately to underlying classical schemes (see the beginning of the proof in Subsection~\ref{subsec:proof-Weibel}), and in both cases this is just the classical affine space.

The proof of Theorem~\ref{thm:Weibel-vanishing} is along the lines of Kerz's proof in the Noetherian case in \cite{KerzICM} which is a bit more direct than the one in \cite[Thm.~B]{KST}. It uses the theory of reductions presented in Subsection~\ref{subsec:reductions} in order to prove better descent results for blowups.
The following lemma replaces \cite[Lemma~4]{KerzStrunk} and is essentially due to Scheiderer~\cite{Scheiderer}.
\begin{lemma}\label{lem1;vanishing} 
Let $X$ be a spectral space, and $\cF$ a sheaf of abelian groups on $X$.
Let $r\geq 0$ be an integer. Assume that $\cF_y=0$ for all points $y\in X$ with $\dim(\ol{\{y\}}) > r$. 
Then $H^n(X,\cF)=0$ for all integers $n>r$.
\end{lemma}

\begin{proof}
This is a direct consequence of results of Scheiderer \cite{Scheiderer}, see the proof of Proposition~4.7 there.
Let $\mathrm{sp}_\bullet(X)\to X$ be the quasi-augmented simplicial topological space defined in \cite[\S 2]{Scheiderer}: Its set of $n$-simplices is given by chains of specialisations 
\[
x_0 \succ \cdots \succ x_n
\]
of points in $X$ with coincidences between the $x_i$ allowed. The topology is induced by the constructible topology on $X^{n+1}$. The quasi-augmentation is given by the canonical map $\mathrm{sp}_0(X) \to X$. Let $\gamma_n: \mathrm{sp}_n(X) \to X$ be the map sending a chain as above to $x_0$. By \cite[Rem.~2.5 and Thm.~4.1]{Scheiderer}, the cohomology groups $H^*(X,\cF)$ are computed by the complex
\[ 
\Gamma(\mathrm{sp}_0(X),\gamma_0^*\cF) \to \Gamma(\mathrm{sp}_1(X),\gamma_1^*\cF) \to \Gamma(\mathrm{sp}_2(X),\gamma_2^*\cF) \to \cdots
\]
which arises from the cosimplicial abelian group 
$[n]\to \Gamma(\mathrm{sp}_n(X), \gamma_n^*\cF)$. Let 
\[ 
B_0 \to B_1\to B_2\to \cdots
\]
be the associated normalized subcomplex, which also computes $H^{*}(X, \cF)$. As in the proof of \cite[Prop.~4.7]{Scheiderer}, $B_n$ coincides with the group of those sections $a \in \Gamma(\mathrm{sp}_n(X), \gamma_n^*\cF)$ whose supports consist only of non-degenerate simplices $x\in \mathrm{sp}_n(X)$.
Hence, it suffices to show $(\gamma_n^*\cF)_x = 0$ for such $x$ if $n>r$.
So let now $x\in \mathrm{sp}_n(X)$ be given by a chain of specialisations $x_0 \succ \cdots \succ x_n$ with pairwise different $x_{i}$'s. Then, we must have $\dim(\overline{\{x_0\}}) \geq n >r$. Thus the assumption implies
$(\gamma_n^*\cF)_x=\cF_{x_0} = 0$ as wanted. This completes the proof of the lemma.
\end{proof}

\begin{lemma}\label{lem2;vanishing} 
Let $X$ be a qcqs spectral scheme of finite valuative dimension. Let $F$ be a sheaf of spectra on $X$ for the Zariski topology.
Assume that for every $x \in X$, the homotopy groups of the stalks $\pi_{-i}(F_{x})$ vanish for $i > \vdim(\cO_{X,x})$.
Then $\pi_{-i}(F(X)) = 0$ for all $i >\vdim(X)$.
\end{lemma}
Note that if $F$ is $K$-theory, then $\pi_{-i}(K_{x})\cong K_{-i}(\cO_{X,x})$ as $K$-theory commutes with filtered colimits.
\begin{proof}
This is deduced from Lemma \ref{lem1;vanishing} by the same argument as the proof of \cite[Prop.~3]{KerzStrunk}. From \cite[Thm.~3.12]{ClausenMathew} we know that the homotopy dimension of $X_{\Zar}$ is bounded by the Krull dimension $\dim(X)$, and in particular the $\infty$-topos of sheaves of spaces on $X_{\Zar}$ is hypercomplete. Hence there is a convergent Zariski descent spectral sequence
\[ 
E_2^{p,q}= H^p(X,\widetilde{F}_{-q}) \Longrightarrow \pi_{-p-q}(F(X)),
\]
where $\widetilde{F}_{-q}$ is the Zariski sheaf of abelian groups on $X$ associated to the presheaf $U \mapsto \pi_{-q}(F(U))$, and $E_{2}^{p,*} = 0$ unless $0\leq p \leq \dim(X)$. 
It suffices to show $E_2^{p,q}=0$ for $p+q> \vdim(X)$. We may assume $p\leq \dim(X)$ so that $q>0$.
Note that $\widetilde F_{-q, x} = \pi_{-q}(F_{x})$. 
By Lemma~\ref{lem1;vanishing} it now suffices to check that $\pi_{-q}(F_{x}) = 0$ for all points $x$ with $\dim(\ol{\{x\}}) > \vdim(X) - q$.
But for such points $x$ we have
\[
\vdim(\cO_{X,x}) \leq \vdim(X) - \dim(\ol{\{x\}})  < q
\] 
by Lemma~\ref{lem;dimv}, and hence $\pi_{-q}(F_{x}) = 0$ by assumption.
\end{proof}

\subsection{$K$-theoretic preparations}

In the proof of Theorem~\ref{thm:Weibel-vanishing}, we use the following  well known facts about non-positive $K$-theory on affine (derived) schemes.
\begin{lemma}		\phantomsection
	\label{lem:properties-negative-K-theory}

\begin{enumerate}
\item Let $A$ be a derived ring. Then the canonical map $K(A) \to K(\pi_{0}(A))$ is 2-connective, i.e.~it induces an isomorphism on $\pi_{i}$ for $i\leq 1$ and a surjection on $\pi_{2}$.

\item Let $A$ be a discrete commutative ring, and let $I \subseteq A$ be a locally nilpotent ideal (i.e., every element of $I$ is nilpotent). Then the map $K(A) \to K(A/I)$ is 1-connective.

\item Let $A$ be a discrete, commutative ring, $X=\Spec(A)$, and let $X = X_{1} \cup X_{2}$ be a closed covering of $X$. Write $X_{12} = X_{1} \cap X_{2}$. Then there is a long exact sequence 
\[
K_{1}(X) \to K_{1}(X_{1}) \oplus K_{1}(X_{2}) \to K_{1}(X_{12}) \to K_{0}(X) \to \dots
\]
\end{enumerate}
\end{lemma}

\begin{proof}
(1) For connective $K$-theory, this is due to Waldhausen~\cite[Prop.~1.1]{Waldhausen}. The general case follows from this together with \cite[Thm.~9.53]{BGT} (or \cite[Thm.~2.16]{KST}). Alternatively, see \cite[Lemma 2.4]{LT} for a slightly more general statement. 

(2) Writing $I$ as a filtered colimit of nilpotent ideals, we may assume that $I$ itself is nilpotent. Then $K_{0}(I) = 0$ by \cite[Exc.~II.2.5]{Weibel} and hence $K_{1}(A) \to K_{1}(A/I)$ is surjective by Prop.~III.2.3 there. Moreover,  $K_{0}(A) \to K_{0}(A/I)$ is an isomorphism by \cite[Lemma~II.2.2]{Weibel}. As the ideal generated by $I$ in any $A$-algebra is still nilpotent, it then follows from the definition of negative $K$-groups \cite[Def.~III.4.1]{Weibel} that $K_{i}(A) \to K_{i}(A/I)$ is an isomorphism for all $i\leq 0$.

(3) Let $I$ and $J$ be the ideals defining $X_{1}$ and $X_{2}$, respectively. Hence $X_{12} = \Spec(A/I+J)$. Let $A' = A/I\cap J$.  Then $A'$ sits in a Milnor square 
\[
\begin{tikzcd}
 A' \ar[d]\ar[r] & A/I \ar[d] \\ 
 A/J \ar[r] & A/I+J. 
\end{tikzcd}
\]
Hence, by \cite[Thm.~III.4.3]{Weibel}, there is a long exact sequence
\[
K_{1}(A') \to K_{1}(A/I) \oplus K_{1}(A/J) \to K_{1}(A/I+J) \to K_{0}(A') \to \dots
\]
As $X_{1} \cup X_{2} = X$, the ideal $I\cap J$ is contained in the nilradical of $A$ and so is locally nilpotent. Using (2), we may thus replace $A'$ by $A$ in the above sequence to get a long exact sequence of the statement.
\end{proof}

The previous lemma has as a simple consequence the following lemma which will allow us to get some control on the negative $K$-theory of certain blowups in Proposition~\ref{prop:blowup-connectivity-non-fp}.

\begin{lemma}
	\label{lem:K-conn-derived-reduced}
Let $X$ be a separated derived scheme admitting a covering by $\ell$ affine open subschemes. Then the canonical morphism
\[
K(X) \to K(\cla X'),
\]
where $\cla X'$ denotes $\cla X$ with any closed subscheme structure between $\cla X^{\red}$ and $\cla X$,
is $(-\ell+2)$-connective.
\end{lemma}

\begin{proof}
Induction on $\ell$. If $\ell =1$, this is Lemma~\ref{lem:properties-negative-K-theory} (1, 2). For $\ell >1$ write $X = V \cup U$ with $U$ affine and $V$ admitting an open covering by $\ell-1$ affines.  As $X$ is separated, also $V \cap U$ has an open covering by $\ell-1$ affines. The claim follows from the inductive hypothesis by considering the map from the cartesian square 
\[
\begin{tikzcd}
 K(X) \ar[d]\ar[r] & K(V) \ar[d] \\ 
 K(U) \ar[r] & K(V\cap U) 
\end{tikzcd}
\]
to the similar cartesian square with the underlying classical schemes with the subscheme structure induced from $X'$.
\end{proof}

In the following, if $X$ is a qcqs derived scheme and $Z \hookrightarrow X$ is a derived subscheme, we denote by $K(X,Z)$ the relative $K$-theory $\fib(K(X) \to K(Z))$.

\begin{lemma}
	\label{lem:descent-finite-non-fp}
Let $X$, $Y$ be  quasi-compact, separated classical schemes, $Z \hookrightarrow X$ a classically finitely presented closed subscheme. Let $f \colon Y \to X$ be a finite morphism which is an isomorphism outside $Z$. Assume that $X$ has a covering by $\ell$ affine open subschemes. Then the canonical map 
\[
K(X,Z) \to K(Y, f^{-1}(Z))
\]
is $(-\ell +1)$-connective. 
\end{lemma}

\begin{proof}
By induction on $\ell$, as in the proof of Lemma~\ref{lem:K-conn-derived-reduced}, it suffices to treat the case $\ell=1$, i.e. $X$ and hence $Y$ affine. 
As $K$-theory commutes with filtered colimits, by Lemma~\ref{lem:approx-finite-ffp} below we may assume that $f\colon Y \to X$ is finitely presented.
By Remark~\ref{rem:pro-descent-for-class-fp}
the square of pro-spectra
\[
\begin{tikzcd}
 K(X) \ar[d]\ar[r] & K(X_{Z}^{\wedge}) \ar[d] \\ 
 K(Y) \ar[r] & K(Y_{Z}^{\wedge}) 
\end{tikzcd}
\]
is then weakly cartesian. Hence the square of spectra obtained by applying $\lim\colon \Pro(\Sp) \to \Sp$,
\[
\begin{tikzcd}
 K(X) \ar[d]\ar[r] & \lim K(X_{Z}^{\wedge}) \ar[d] \\ 
 K(Y) \ar[r] & \lim K(Y_{Z}^{\wedge}),
\end{tikzcd}
\]
is cartesian.
By Lemma~\ref{lem:properties-negative-K-theory}(1, 2), each transition map in the pro-spectrum $K(X_{Z}^{\wedge})$ is 1-connective. In particular, using the Milnor sequence we get isomorphisms $\pi_{0}(\lim K(X_{Z}^{\wedge})) \cong \lim K_{0}(X_{Z}^{\wedge}) \cong K_{0}(Z)$ and surjections $\pi_{1}(\lim K(X_{Z}^{\wedge})) \to \lim K_{1}(X_{Z}^{\wedge}) \to K_{1}(Z)$. In other words, $\lim K(X_{Z}^{\wedge}) \to K(Z)$ is 1-connective. Consider the following diagram in which the left vertical map is an equivalence by the previous cartesian square.
\[
\begin{tikzcd}
 \fib(K(X) \to \lim K(X_{Z}^{\wedge})) \ar[d, "\simeq"]\ar[r] & K(X,Z) \ar[d] \\ 
\fib(K(Y) \to \lim K(Y_{Z}^{\wedge}) ) \ar[r] & K(Y, f^{-1}(Z)) 
\end{tikzcd}
\]
The fibre of the top horizontal map identifies with $\Omega \fib( \lim K(X_{Z}^{\wedge}) \to K(Z))$ and is hence 0-connective.
By the same arguments, the lower horizontal map is 0-connective. It follows that also the right vertical map is 0-connective.
\end{proof}

\begin{lemma}
	\label{lem:approx-finite-ffp}
Let $\varphi \colon A \to B$ be a finite morphism of rings such that $\Spec(\varphi)\colon \Spec(B) \to \Spec(A)$ is an isomorphism outside $V(I)$ where $I = (f_{1}, \dots, f_{r}) \subseteq A$. Then we can write $\varphi$ as a filtered colimit of finitely presented, finite maps $\varphi_{\lambda} \colon A \to B_{\lambda}$ such that each $\Spec(\varphi_{\lambda})$ is an isomorphism outside $V(I)$.
\end{lemma}

\begin{proof}
Choose $A$-module generators $b_{1}, \dots, b_{s}$ of $B$. We get a surjection $A[T_{1}, \dots, T_{s}] \to B$, $T_{i} \mapsto b_{i}$. Let $J$ denote its kernel. As $b_{i}$ is integral over $A$, there exists a monic polynomial $p_{i} \in A[X]$ with $p_{i}(b_{i})=0$. Clearly, $p_{i}(T_{i}) \in J$. 

Let $f$ be one of the generators $f_{j}$ of $I$. As $\Spec(\varphi)$ is an isomorphism outside $V(I)$, we have $A[f^{-1}] \xrightarrow{\simeq} B[f^{-1}]$. In particular, $B[f^{-1}]$ is finitely presented as an $A[f^{-1}]$-algebra, and it follows that $J[f^{-1}] \subseteq A[f^{-1}][T_{1}, \dots, T_{s}]$ is a finitely generated ideal. We may pick finitely many generators $q_{jk}$ which already live in $J$. 

Now write $J$ as increasing union of finitely generated ideals $J_{\lambda}$ where each $J_{\lambda}$ contains all the $p_{i}(T_{i})$ and all $q_{jk}$'s. By construction, $B_{\lambda} = A[T_{1}, \dots, T_{s}]/J_{\lambda}$ is then finitely presented over $A$, finite (as in $B_{\lambda}$ each $T_{i}$ is integral  over $A$), and $\Spec(B_{\lambda}) \to \Spec(A)$ is an isomorphism outside $V(I)$ (as $J_{\lambda}[f_{j}^{-1}] = J[f_{j}^{-1}]$ for all $j$). Clearly, $\colim_{\lambda} B_{\lambda} = B$.
\end{proof}

\begin{prop}
	\label{prop:blowup-connectivity-non-fp}
Let $R$ be a local ring, $I \subseteq R$ be a finitely generated, proper ideal of analytic spread $\ell=\ell(I)$. Let $Z = V(I)$,  $f\colon \Bl_{Z}(X) \to X$ the blowup, and $E = f^{-1}(Z)$ the classical exceptional divisor. Then the pullback map 
\[
K(X, Z) \to K(\Bl_{Z}(X), E) 
\]
is $(-\ell+1)$-connective. More generally, for any $r\geq 0$, the pullback map 
\[
K(\mathbb A^{r} \times X, \mathbb A^{r} \times Z) \to K(\mathbb A^{r} \times  \Bl_{Z}(X), \mathbb A^{r} \times  E) 
\]
is $(-\ell+1)$-connective.
\end{prop}

\begin{proof}
By Proposition~\ref{prop:existence-reduction}, there exists a reduction $J$ of some power $I^{n}$ of $I$ generated by $\ell$ elements, say $J= (f_{1}, \dots, f_{\ell})$. As $\Bl_{I^{n}}(R) = \Bl_{I}(R)$ canonically, Lemma~\ref{lem:reduction-finiteness} yields a finite $X$-morphism $\Bl_{Z}(X)=\Bl_{I}(R) \to \Bl_{J}(R)$. Let $\widetilde X$ be the derived blowup of $X$ in $f_{1}, \dots, f_{\ell}$. Then we have a closed immersion $\Bl_{J}(R) \hookrightarrow \widetilde X$ and in total a finite morphism $\Bl_{I}(X) \to \widetilde X$. Moreover, $\widetilde X$ has a covering by $\ell$ affine open subschemes. Write $D = \cla(\widetilde X \times_{X} Z)$. 
We now argue similarly as in the proof of Lemma~\ref{lem:descent-finite-non-fp}. Consider the following commutative diagram.
\[
\begin{tikzcd}
 \fib(K(X) \to \lim K(X_{Z}^{\wedge})) \ar[d, "\simeq"]\ar[r] & K(X,Z) \ar[d] \\ 
 \fib(K(\widetilde X) \to \lim K(\widetilde X_{Z}^{\wedge})) \ar[r] & K(\widetilde X, D) 
\end{tikzcd}
\]
The left vertical map is an equivalence by Proposition~\ref{prop:descent-derived-bu} (note that $|Z| = |V(J)|$). The top horizontal map is 0-connective, the lower horizontal map is $(-\ell+1)$-connective by Lemma~\ref{lem:K-conn-derived-reduced}. Hence the right vertical map is $(-\ell+1)$-connective. Using Lemma~\ref{lem:K-conn-derived-reduced} again, we see that $K(\widetilde X, D) \to K(\cla {\widetilde X}, D)$ is $(-\ell+2)$-connective. In total, 
$K(X,Z) \to K(\cla {\widetilde X}, D)$ is $(-\ell+1)$-connective.\footnote{In fact, using the above connectivity estimates, this follows more easily from the derived generalization of Thomason's theorem \cite{ThomasonBlowup} in \cite[Thm.~3.7]{KST} which holds without any Noetherianity assumptions (with the same proof as in \emph{loc.~cit.})} By Lemma~\ref{lem:descent-finite-non-fp}, the map $K(\cla {\widetilde X}, D) \to K(\Bl_{Z}(X), E)$ is $(-\ell+1)$-connective. Hence the composite is $(-\ell+1)$-connective, as asserted.

The more general claim can be proved by exactly the same argument, using that $\mathbb A^{r}$ is affine and flat over $\Z$ and derived blowups commute with arbitrary base change.
\end{proof}

\subsection{Proof of Weibel vanishing}
	\label{subsec:proof-Weibel}

\begin{proof}[Proof of Theorem~\ref{thm:Weibel-vanishing}]
We write $N^{(r)}K(X)$ for the cofiber of the canonical split inclusion $K(X) \to K(\A^{r}_{X})$ so that we have $K(\A^{r}_{X}) \cong K(X) \oplus N^{(r)}K(X)$.
Then assertion (2) of the theorem is equivalent to the statement that $N^{(r)}K_{-i}(X)=0$ for all $i \geq \vdim(X)$.

We may assume that $d = \vdim(X)$ is finite and prove the theorem  by induction on $d$.
By Lemma~\ref{lem2;vanishing}, applied to $K$ and $\Sigma N^{(r)}K$ respectively, we may assume that $X$ is affine and local. In this case, $K_{-i}(\A^{r}_{X}) \cong K_{-i}(\A^{r}_{\cla X^{\red}})$ for  all $i\geq 0$ and all $r\geq 0$ by Lemma~\ref{lem:properties-negative-K-theory} (1), (2). So we can assume that $X$ is a classical reduced affine local scheme. 

If $d=0$, then also the Krull dimension $\dim(X)$ is $0$. As $X$ is local and reduced, it is the spectrum of a field. Hence $K_{-i}(X) = 0$ and for all $i>0$ and $N^{(r)}K(X) = 0$ (by \cite[Thm.~II.7.8]{Weibel} and the definition of negative $K$-groups), as $X$ is regular Noetherian. 

Now assume that $d$ is positive and the assertion is proven for all qcqs derived schemes that are of valuative dimension $< d$. 

\begin{claim}
The assertion of Theorem~\ref{thm:Weibel-vanishing} holds for $X$ as above under the additional assumption that $X$ is irreducible.
\end{claim}

\begin{proof}[Proof of the claim.]
Write $X=\Spec(R)$. So $R$ is a local integral domain.
Let  $\gamma$ be an element in $K_{-i}(R)$ with $i > d$ or in $N^{(r)}K_{-i}(R) \subseteq K_{-i}(\mathbb A^{r}_{R})$ with $i \geq d$ in case $r>0$. We have to show that $\gamma=0$.

As $K$-theory commutes with filtered colimits, there exists a subring $R_{0}$ of $R$, finitely generated as a $\Z$-algebra, and $\gamma_{0} \in K_{-i}(\mathbb A^{r}_{R_{0}})$ pulling back to $\gamma$. By \cite[Prop.~5]{KerzStrunk}, there exists a (finitely generated) ideal $I_{0} \subseteq R_{0}$ such that $\gamma_{0}$ is annihilated in $K_{-i}(\mathbb A^{r} \times \Bl_{I_{0}}(R_{0}))$. Let $I=RI_{0} \subseteq R$, and $Z=V(I)$. By functoriality of the blowup, $\gamma$ is annihilated in $K_{-i}(\mathbb A^{r} \times  \Bl_{Z}(X))$. Let $E \subseteq \Bl_{Z}(X)$ denote the exceptional divisor. By construction $Z \subseteq X$ is a finitely presented, proper subscheme, hence $\Bl_{Z}(X) \to X$ is a modification and thus $\vdim \Bl_{Z}(X) = \vdim X = d$ \cite[Prop.~2.3.2(6)]{EHIK}. Furthermore $\vdim(Z), \vdim(E) < d$ by \emph{loc.~cit.}, Assertion~4.

We now treat the case of $\gamma\in K_{-i}(R)$ with $i>d$. Consider the commutative diagram of exact sequences
\[
\begin{tikzcd}
K_{-i+1}(Z) \ar[r]\ar[d]  & K_{-i}(X,Z) \ar[r]\ar[d] & K_{-i}(X) \ar[r]\ar[d] & K_{-i}(Z)\ar[d] \\
K_{-i+1}(E) \ar[r]  & K_{-i}(\Bl_{Z}(X), E) \ar[r] & K_{-i}(\Bl_{Z}(X)) \ar[r] & K_{-i}(E).
\end{tikzcd}
\]
As $\vdim(Z), \vdim(E) < d$, the four outer groups vanish by induction. If $I=R$, then $\Bl_{Z}(X) = X$ and hence $\gamma=0$. So we may now assume that $I \subset R$ is a proper ideal.
By Proposition~\ref{prop:analytic-spread-vdim}  we then have $\ell(I) \leq \vdim(R)=d$. As $i > d$,  Proposition~\ref{prop:blowup-connectivity-non-fp} implies that the second vertical map is an isomorphism. Hence $K_{-i}(X) \to K_{-i}(\Bl_{Z}(X))$ is an isomorphism, too, and hence $\gamma=0$.

The argument in the case $\gamma\in N^{(r)}K_{-i}(R)$ with $i \geq d$ is exactly the same with the $K$-groups above replaced by the $N^{(r)}K$-groups. This finishes the proof of the Claim.
\end{proof}

We now consider the general case. So $X$ is now a classical, reduced, affine, local scheme of valuative dimension $d>0$. 
For ease of notation, we only consider assertion (1) of the theorem. The proof of (2) is completely parallel.
To reduce to the integral case treated in the Claim above, we apply an argument from \cite[Thm. 2.4.15]{EHIK} as follows: 
Take an integer $i > d=\vdim(X)$ and an element $\gamma\in K_{-i}(X)$. We wish to show that $\gamma=0$. 
Let 
\[ 
\cE= \{Z \hookrightarrow X \;\text{reduced, closed}\; | \;\gamma_{|Z}\not=0\in K_{-i}(Z)\}.
\]
We need to prove that $\cE =\emptyset$. Note that every $Z \in \cE$ is itself a reduced local affine scheme and that $\cE$ is ordered by inclusion. 
Let $(Z_\lambda)_{\lambda\in \Lambda}$ be a descending chain in $\cE$ and put 
$Z = \lim_{\lambda\in\Lambda}Z_\lambda$.
As $K$-theory commutes with filtered colimits of rings, we have 
$K_{-i}(Z) = \colim_{\lambda\in \Lambda} K_{-i}(Z_\lambda)$. 
So if $\gamma_{|Z}=0$, then there exists a $\lambda$ such that $\gamma_{|Z_\lambda}=0$, which is a contradiction. 
Hence $\gamma_{|Z}\not= 0$ so that $Z\in\cE$ and $Z$ is a lower bound of 
$(Z_\lambda)_{\lambda\in \Lambda}$. If $\cE\not=\emptyset$, we may apply Zorn's lemma to conclude that $\cE$ has a minimal element $Z$. As $K_{-i}(Z)\not=0$, $Z$ must be reducible. 
Let $Z^{\gen}$ be the set of the generic points of $Z$ equipped with the induced topology from the underlying topological space of $Z$.
By \cite[Cor.~2.4]{HenriksenJerison} (see also \cite[Lemma 2.4.14]{EHIK}), there exists a decomposition $Z^{\gen}=S_1\sqcup S_2$ with $S_i$ closed and non-empty for $i=1,2$. Letting $Z_i$ be the closure of $S_i$ in $Z$ with reduced scheme structure, we have $Z = Z_1\cup Z_2$ and $Z_i\cap Z^{\gen}=S_i$ for $i=1,2$. In particular, $Z_1\cap Z_2\cap Z^{\gen}=\emptyset$ so that $\vdim(Z_1\cap Z_2) < d$ by \cite[Prop.~2.3.2(4)]{EHIK}.
By excision in non-positive $K$-theory for closed coverings of affine schemes (Lemma~\ref{lem:properties-negative-K-theory}(3)), we have an exact sequence
\[ 
K_{-i+1}(Z_1 \cap Z_2) \to K_{-i}(Z) \to K_{-i}(Z_1) \oplus K_{-i}(Z_2).
\]
The group on the left-hand side vanishes by induction. As $Z$ was minimal in $\cE$, we must have
$\gamma_{| Z_i}=0$ for $i = 1, 2$. Thus we get $\gamma_{|Z}=0$, which is a contradiction. Thus we must have $\cE =\emptyset$, which completes the proof of Theorem \ref{thm:Weibel-vanishing}.
 \end{proof}

We now give a proof of Theorem~\ref{thm:Weibel-vanishing-KH}, 
reproduced from \cite[Rem.~3.5]{KellyMorrow}.
We refer to \cite[\S 2.1]{EHIK} for a discussion of the cdh topology in this generality but note that the definition of the cdh topology used in \cite{EHIK} is the one used by Suslin and Voevodsky, \cite[Def.~5.7]{SVBK}, and the proof that cdh \v{C}ech descent is equivalent to cdh excision is Voevodsky's proof from \cite{VoeHTS}, rewritten in modern language in \cite[Thm.~3.2.5]{AHW17}. Voevodsky's proof that cdh \v{C}ech descent is equivalent to cdh hyperdescent requires Noetherian hypotheses that were lifted in \cite{EHIK}.

\begin{proof}[Proof of Theorem~\ref{thm:Weibel-vanishing-KH}]
We note that homotopy $K$-theory is a cdh sheaf. This was first proven by Cisinski \cite{Cisinski} for Noetherian schemes of finite dimension (and in characteristic 0 previously by Haesemeyer \cite{Haesemeyer}), which implies the general statement by absolute Noetherian approximation; alternatively it follows from the fact that $KH$ is truncating, see \cite[Cor.~A.5]{LT}. 
In particular, we get the maps 
\begin{equation} \label{equa:LKLKKH}
L_{\cdh}K_{\geq 0} \to L_{\cdh}K \to KH
\end{equation}
which we want to show are equivalences.
By \cite[Thm.~2.4.15, Cor.~2.3.3]{EHIK} the $\infty$-topos of cdh sheaves of spaces on finitely presented $X$-schemes is locally of finite homotopy dimension and of homotopy dimension $\leq \vdim(X)$. 
This implies that we get a convergent spectral sequence 
\begin{equation} \label{equa:EKHcdh}
E_{2}^{p,q} = H^{p}(X_{\mathrm{cdh}}, \tilde\pi_{-q}KH) \Longrightarrow KH_{-p-q}(X). 
\end{equation}
where $\tilde\pi_{q}KH$ denotes the cdh sheafified homotopy groups of $KH$. 
A conservative family of points for the cdh topology is given by the spectra of henselian valuation rings \cite{GL01}, \cite[Thm.~2.3, Thm.~2.6]{GK15}, \cite[Cor.~2.4.19]{EHIK}. As  $K$-theory of valuation rings is connective and agrees with its homotopy $K$-theory \cite[Thm.~3.4]{KellyMorrow}, \cite[Lem.~4.3]{KSTVorst}, and $K$ and $KH$ commute with filtered colimits, the maps \eqref{equa:LKLKKH} are equivalences
so we get part (2) of Theorem~\ref{thm:Weibel-vanishing-KH}. The vanishing 
\[ KH_{-i}(X) = 0, \qquad \textrm{ for all } i > \vdim(X) \]
claimed in part (1) follows from the spectral sequence \eqref{equa:EKHcdh}.
\end{proof}

\bibliographystyle{amsalpha}
\bibliography{cochains}

\providecommand{\bysame}{\leavevmode\hbox to3em{\hrulefill}\thinspace}
\providecommand{\MR}{\relax\ifhmode\unskip\space\fi MR }
\providecommand{\MRhref}[2]{%
  \href{http://www.ams.org/mathscinet-getitem?mr=#1}{#2}
}
\providecommand{\href}[2]{#2}
\begin{thebibliography}{CMNN20}

\bibitem[AHW17]{AHW17}
Aravind Asok, Marc Hoyois, and Matthias Wendt, \emph{{Affine representability
  results in ${\mathbb{A}}^{1}$-homotopy theory, I: Vector bundles}}, Duke
  Mathematical Journal \textbf{166} (2017), no.~10, 1923 -- 1953.

\bibitem[Ann22]{AnnalaBase}
Toni Annala, \emph{Base independent algebraic cobordism}, J. Pure Appl. Algebra
  \textbf{226} (2022), no.~6, Paper No. 106977, 44. \MR{4346005}

\bibitem[Ant18]{Antieau-notes}
Ben Antieau, \emph{Various remarks on {W}eibel's conjecture}, Letter to Kerz,
  June 2018.

\bibitem[BGT13]{BGT}
Andrew~J. Blumberg, David Gepner, and Gon\c{c}alo Tabuada, \emph{A universal
  characterization of higher algebraic {$K$}-theory}, Geom. Topol. \textbf{17}
  (2013), no.~2, 733--838. \MR{3070515}

\bibitem[BKRS22]{BachmannCatMilnor}
Tom Bachmann, Adeel~A. Khan, Charanya Ravi, and Vladimir Sosnilo,
  \emph{Categorical {M}ilnor squares and {K}-theory of algebraic stacks},
  Selecta Math. (N.S.) \textbf{28} (2022), no.~5, Paper No. 85, 72.
  \MR{4487748}

\bibitem[Bou24]{Bouis-thesis}
Tess Bouis, \emph{On the motivic cohomology of mixed characteristic schemes},
  Ph.D. thesis, Universit{\'e} Paris-Saclay, 2024.

\bibitem[BZFN10]{BZFN}
David Ben-Zvi, John Francis, and David Nadler, \emph{Integral transforms and
  {D}rinfeld centers in derived algebraic geometry}, J. Amer. Math. Soc.
  \textbf{23} (2010), no.~4, 909--966. \MR{2669705}

\bibitem[CHSW08]{MR2415380}
Guillermo Corti{\~n}as, Christian Haesemeyer, Marco Schlichting, and Charles
  Weibel, \emph{Cyclic homology, cdh-cohomology and negative {$K$}-theory},
  Ann. of Math. (2) \textbf{167} (2008), no.~2, 549--573.

\bibitem[Cis13]{Cisinski}
Denis-Charles Cisinski, \emph{Descente par \'{e}clatements en {$K$}-th\'{e}orie
  invariante par homotopie}, Ann. of Math. (2) \textbf{177} (2013), no.~2,
  425--448. \MR{3010804}

\bibitem[CM21]{ClausenMathew}
Dustin Clausen and Akhil Mathew, \emph{Hyperdescent and \'{e}tale
  {$K$}-theory}, Invent. Math. \textbf{225} (2021), no.~3, 981--1076.
  \MR{4296353}

\bibitem[CMNN20]{CMNN}
Dustin Clausen, Akhil Mathew, Niko Naumann, and Justin Noel, \emph{Descent in
  algebraic {$K$}-theory and a conjecture of {A}usoni-{R}ognes}, J. Eur. Math.
  Soc. (JEMS) \textbf{22} (2020), no.~4, 1149--1200. \MR{4071324}

\bibitem[Cor06]{Cortinas}
Guillermo Corti{\~n}as, \emph{The obstruction to excision in {$K$}-theory and
  in cyclic homology}, Invent. Math. \textbf{164} (2006), no.~1, 143--173.

\bibitem[DT22]{Dahlhausen}
Christian Dahlhausen and Georg Tamme, \emph{A counterexample to pro-cdh descent
  for non-noetherian schemes},
  \url{https://cdahlhausen.eu/Notes/pro-cdh-counterex.pdf}, 2022.

\bibitem[Efi24]{Efimov:2024aa}
Alexander~I. Efimov, \emph{K-theory and lozalizing invariants of large
  categories}, arXiv:2405.12169, 2024.

\bibitem[EHIK20]{EHIK}
Elden Elmanto, Marc Hoyois, Ryomei Iwasa, and Shane Kelly, \emph{Cdh descent,
  cdarc descent, and {M}ilnor excision}, Mathematische Annalen \textbf{379}
  (2020), no.~3--4, 1011--1045.

\bibitem[EM23]{ElmantoMorrow}
Elden Elmanto and Matthew Morrow, \emph{Motivic cohomology of
  equicharacteristic schemes}, arXiv:2309.08463, 2023.

\bibitem[GD71]{EGAInew}
Alexander Grothendieck and Jean~A. Dieudonn\'{e}, \emph{\'{E}l\'{e}ments de
  g\'{e}om\'{e}trie alg\'{e}brique. {I}}, Grundlehren der mathematischen
  Wissenschaften [Fundamental Principles of Mathematical Sciences], vol. 166,
  Springer-Verlag, Berlin, 1971. \MR{3075000}

\bibitem[GH06]{GH}
Thomas Geisser and Lars Hesselholt, \emph{Bi-relative algebraic {$K$}-theory
  and topological cyclic homology}, Invent. Math. \textbf{166} (2006), no.~2,
  359--395.

\bibitem[GH11]{GH2}
\bysame, \emph{On relative and bi-relative algebraic {$K$}-theory of rings of
  finite characteristic}, J. Amer. Math. Soc. \textbf{24} (2011), no.~1,
  29--49.

\bibitem[GK15]{GK15}
Ofer Gabber and Shane Kelly, \emph{Points in algebraic geometry}, J. Pure Appl.
  Algebra \textbf{219} (2015), no.~10, 4667--4680. \MR{3346512}

\bibitem[GL01]{GL01}
Thomas~G. Goodwillie and Stephen Lichtenbaum, \emph{A cohomological bound for
  the {$h$}-topology}, Amer. J. Math. \textbf{123} (2001), no.~3, 425--443.
  \MR{1833147}

\bibitem[GR14]{MR3220628}
Dennis Gaitsgory and Nick Rozenblyum, \emph{D{G} indschemes}, Perspectives in
  representation theory, Contemp. Math., vol. 610, Amer. Math. Soc.,
  Providence, RI, 2014, pp.~139--251. \MR{3220628}

\bibitem[GR17]{GRII}
\bysame, \emph{A study in derived algebraic geometry. {V}ol. {II}.
  {D}eformations, {L}ie theory and formal geometry}, Mathematical Surveys and
  Monographs, vol. 221, American Mathematical Society, Providence, RI, 2017.
  \MR{3701353}

\bibitem[Gro61]{EGAIII}
A.~Grothendieck, \emph{\'{E}l\'{e}ments de g\'{e}om\'{e}trie alg\'{e}brique.
  {III}. \'{E}tude cohomologique des faisceaux coh\'{e}rents. {I}}, Inst.
  Hautes \'{E}tudes Sci. Publ. Math. (1961), no.~11, 167. \MR{217085}

\bibitem[Gro66]{EGAIV3}
Alexander Grothendieck, \emph{\'{E}l\'{e}ments de g\'{e}om\'{e}trie
  alg\'{e}brique. {IV}. \'{E}tude locale des sch\'{e}mas et des morphismes de
  sch\'{e}mas. {III}}, Inst. Hautes \'{E}tudes Sci. Publ. Math. (1966), no.~28,
  255. \MR{217086}

\bibitem[Hae04]{Haesemeyer}
Christian Haesemeyer, \emph{Descent properties of homotopy {$K$}-theory}, Duke
  Math. J. \textbf{125} (2004), no.~3, 589--620.

\bibitem[Hek21]{Hekking:2021aa}
Jeroen Hekking, \emph{Graded algebras, projective spectra and blow-ups in
  derived algebraic geometry}, arXiv:2106.01270, 2021.

\bibitem[HJ65]{HenriksenJerison}
Melvin Henriksen and Meyer Jerison, \emph{The space of minimal prime ideals of
  a commutative ring}, Trans. Amer. Math. Soc. \textbf{115} (1965), 110--130.
  \MR{194880}

\bibitem[HS06]{SwansonHuneke}
Craig Huneke and Irena Swanson, \emph{Integral closure of ideals, rings, and
  modules}, London Mathematical Society Lecture Note Series, vol. 336,
  Cambridge University Press, Cambridge, 2006. \MR{2266432}

\bibitem[Jaf60]{Jaffard}
Paul Jaffard, \emph{Th\'{e}orie de la dimension dans les anneaux de polynomes},
  M\'{e}mor. Sci. Math., Fasc. 146, Gauthier-Villars, Paris, 1960. \MR{117256}

\bibitem[Kel24]{Kelly:2024ab}
Shane Kelly, \emph{Non-reduced valuation rings and descent for smooth blowup
  squares}, arXiv:2401.02706, 2024.

\bibitem[Ker18]{KerzICM}
Moritz Kerz, \emph{On negative algebraic {$K$}-groups}, Proceedings of the
  {I}nternational {C}ongress of {M}athematicians---{R}io de {J}aneiro 2018.
  {V}ol. {II}. {I}nvited lectures, World Sci. Publ., Hackensack, NJ, 2018,
  pp.~163--172. \MR{3966761}

\bibitem[Kha20]{KhanSemiOrthogonal}
Adeel~A. Khan, \emph{Algebraic {K}-theory of quasi-smooth blow-ups and {CDH}
  descent}, Ann. H. Lebesgue \textbf{3} (2020), 1091--1116. \MR{4149835}

\bibitem[KM21]{KellyMorrow}
Shane Kelly and Matthew Morrow, \emph{{$K$}-theory of valuation rings}, Compos.
  Math. \textbf{157} (2021), no.~6, 1121--1142. \MR{4264079}

\bibitem[KR18]{KhanRydh}
Adeel~A. Khan and David Rydh, \emph{Virtual cartier divisors and blow-ups},
  arXiv:1802.05702, 02 2018.

\bibitem[Kri10]{MR2629600}
Amalendu Krishna, \emph{An {A}rtin-{R}ees theorem in {$K$}-theory and
  applications to zero cycles}, J. Algebraic Geom. \textbf{19} (2010), no.~3,
  555--598. \MR{2629600}

\bibitem[KS02]{MR1935844}
Amalendu Krishna and Vaseduvan Srinivas, \emph{Zero-cycles and {$K$}-theory on
  normal surfaces}, Ann. of Math. (2) \textbf{156} (2002), no.~1, 155--195.
  \MR{1935844}

\bibitem[KS17]{KerzStrunk}
Moritz Kerz and Florian Strunk, \emph{On the vanishing of negative homotopy
  {$K$}-theory}, J. Pure Appl. Algebra \textbf{221} (2017), no.~7, 1641--1644.
  \MR{3614971}

\bibitem[KS24]{Kelly:2024aa}
Shane Kelly and Shuji Saito, \emph{A procdh topology}, arXiv:2401.02699, 2024.

\bibitem[KST18]{KST}
Moritz Kerz, Florian Strunk, and Georg Tamme, \emph{Algebraic {$K$}-theory and
  descent for blow-ups}, Invent. Math. \textbf{211} (2018), no.~2, 523--577.

\bibitem[KST19]{KST-non-arch}
Moritz Kerz, Shuji Saito, and Georg Tamme, \emph{{$K$}-theory of
  non-{A}rchimedean rings. {I}}, Doc. Math. \textbf{24} (2019), 1365--1411.
  \MR{4012551}

\bibitem[KST21]{KSTVorst}
Moritz Kerz, Florian Strunk, and Georg Tamme, \emph{Towards {V}orst's
  conjecture in positive characteristic}, Compos. Math. \textbf{157} (2021),
  no.~6, 1143--1171. \MR{4270122}

\bibitem[KST23]{MR4637146}
Moritz Kerz, Shuji Saito, and Georg Tamme, \emph{{$K$}-theory of
  non-{A}rchimedean rings {II}}, Nagoya Math. J. \textbf{251} (2023), 669--685.
  \MR{4637146}

\bibitem[LT19]{LT}
Markus Land and Georg Tamme, \emph{On the {$K$}-theory of pullbacks}, Ann. of
  Math. (2) \textbf{190} (2019), no.~3, 877--930.

\bibitem[Lur04]{LurieThesis}
Jacob Lurie, \emph{Derived algebraic geometry}, Ph.D. thesis, Massachusetts
  Institute of Technology. Dept. of Mathematics., 2004.

\bibitem[Lur17]{HA}
\bysame, \emph{Higher {A}lgebra}, Available at the author's homepage
  \url{http://www.math.ias.edu/~lurie}, 2017.

\bibitem[Lur18]{SAG}
\bysame, \emph{Spectral {A}lgebraic {G}eometry}, Available at the author's
  homepage \url{http://www.math.ias.edu/~lurie}, 2018.

\bibitem[Mor16a]{Morrow:2016aa}
Matthew Morrow, \emph{A historical overview of pro cdh descent in algebraic
  {$K$}-theory and its relation to rigid analytic varieties}, arXiv:1612.00418,
  2016.

\bibitem[Mor16b]{Mor}
\bysame, \emph{Pro {CDH}-descent for cyclic homology and {$K$}-theory}, J.
  Inst. Math. Jussieu \textbf{15} (2016), no.~3, 539--567.

\bibitem[Mor18]{Morrow-pro-unitality}
\bysame, \emph{Pro unitality and pro excision in algebraic {$K$}-theory and
  cyclic homology}, J. Reine Angew. Math. \textbf{736} (2018), 95--139.

\bibitem[Mor23]{MorrowOberwolfach}
\bysame, \emph{Thoughts on {W}eibel's conjecture \emph{in} {Algebraic}
  {K}-{T}heory}, Oberwolfach Reports \textbf{19} (2023), no.~2, 1354--1356.

\bibitem[RG71]{Raynaud-Gruson}
Michel Raynaud and Laurent Gruson, \emph{Crit{\`e}res de platitude et de
  projectivit\'{e}. {T}echniques de ``platification'' d'un module}, Invent.
  Math. \textbf{13} (1971), 1--89. \MR{308104}

\bibitem[Sch]{ScholzeCondensed}
Peter Scholze, \emph{Lectures on condensed mathematics},
  \url{https://www.math.uni-bonn.de/people/scholze/Condensed.pdf}.

\bibitem[Sch92]{Scheiderer}
Claus Scheiderer, \emph{Quasi-augmented simplicial spaces, with an application
  to cohomological dimension}, J. Pure Appl. Algebra \textbf{81} (1992), no.~3,
  293--311. \MR{1179103}

\bibitem[{Sta}25]{stacks}
The {Stacks project authors}, \emph{The stacks project},
  \url{https://stacks.math.columbia.edu}, 2025.

\bibitem[SV00]{SVBK}
Andrei Suslin and Vladimir Voevodsky, \emph{{B}loch-{K}ato conjecture and
  motivic cohomology with finite coefficients}, The arithmetic and geometry of
  algebraic cycles, Springer, 2000, pp.~117--189.

\bibitem[Tho93]{ThomasonBlowup}
Robert~W. Thomason, \emph{Les {$K$}-groupes d'un sch\'{e}ma \'{e}clat\'{e} et
  une formule d'intersection exc\'{e}dentaire}, Invent. Math. \textbf{112}
  (1993), no.~1, 195--215. \MR{1207482}

\bibitem[TV05]{MR2137288}
Bertrand To\"{e}n and Gabriele Vezzosi, \emph{Homotopical algebraic geometry.
  {I}. {T}opos theory}, Adv. Math. \textbf{193} (2005), no.~2, 257--372.
  \MR{2137288}

\bibitem[TV08]{MR2394633}
\bysame, \emph{Homotopical algebraic geometry. {II}. {G}eometric stacks and
  applications}, Mem. Amer. Math. Soc. \textbf{193} (2008), no.~902, x+224.
  \MR{2394633}

\bibitem[Voe10]{VoeHTS}
Vladimir Voevodsky, \emph{Homotopy theory of simplicial sheaves in completely
  decomposable topologies}, Journal of pure and applied algebra \textbf{214}
  (2010), no.~8, 1384--1398.

\bibitem[Wal78]{Waldhausen}
Friedhelm Waldhausen, \emph{Algebraic {K}-theory of topological spaces. {I}},
  Proceedings of Symposia in Pure Mathematics \textbf{32} (1978).

\bibitem[Wei13]{Weibel}
Charles Weibel, \emph{The {$K$}-book}, Graduate Studies in Mathematics, vol.
  145, American Mathematical Society, Providence, RI, 2013, An introduction to
  algebraic $K$-theory.

\end{thebibliography}

\end{document}